\documentclass[10pt]{amsart}
\usepackage{amssymb}
\usepackage{amsfonts}
\usepackage{amsthm}
\usepackage{mathtools}
\usepackage{todonotes}

\usepackage{enumitem} 

\calclayout

\usepackage{amsthm}

\usepackage[pagebackref,colorlinks,citecolor=blue,linkcolor=blue]{hyperref}

\usepackage{cite}   

\usepackage{amsmath,bbm,amssymb,amsxtra}

\def\Xint#1{\mathchoice
	{\XXint\displaystyle\textstyle{#1}}
	{\XXint\textstyle\scriptstyle{#1}}
	{\XXint\scriptstyle\scriptscriptstyle{#1}}
	{\XXint\scriptscriptstyle\scriptscriptstyle{#1}}
	\!\int}

\def\XXint#1#2#3{{\setbox0=\hbox{$#1{#2#3}{\int}$ }
		\vcenter{\hbox{$#2#3$ }}\kern-.6\wd0}}
\def\dashint{\Xint-}

\newtheorem{theorem}{Theorem}[section]
\newtheorem{lemma}[theorem]{Lemma}
\newtheorem{proposition}[theorem]{Proposition}

\newtheorem{corollary}[theorem]{Corollary}
\newtheorem{definition}[theorem]{Definition}
\newtheorem{example}[theorem]{Example}
\newtheorem{question}[theorem]{Question}
\numberwithin{equation}{section}
\newcommand{\Mod}{{\rm Mod}}
\newcommand{\N}{\mathbb{N}}
\newcommand{\Ha}{\mathcal{H}}
\newcommand{\R}{\mathbb{R}}
\newcommand{\wtil}{\widetilde}
\newcommand{\diam}{\text{\rm diam}}
\newcommand{\rad}{\text{\rm rad}}

\title[Large-scale behaviour of Sobolev functions]{Large-scale behaviour of Sobolev functions in Ahlfors regular metric measure spaces}

\author{Josh Kline}
\address{Department of Mathematical Sciences, P.O. Box 210025, University of Cincinnati,
Cincinnati, OH 45221–0025, U.S.A.}
\email{klinejp@ucmail.uc.edu}

\author{Pekka Koskela}
\address{Department of Mathematics and Statistics \\
P.O. Box 35 \\
FI-40014 University of Jyväskylä, Finland}
\email{pekka.j.koskela@jyu.fi}

\author{Khanh Nguyen}
\address{Academy of Mathematics and Systems Science, Chinese Academy of Sciences, Beijing 100190, PR China}
\email{khanhnguyen@amss.ac.cn}

\date{\today}
\usepackage{datetime}
\begin{document}
\maketitle
\begin{abstract}
In this paper, we study the behaviour at infinity of $p$-Sobolev functions in the setting of Ahlfors $Q$-regular metric measure spaces supporting a $p$-Poincar\'e inequality.  By introducing the notions of sets which are $p$-thin at infinity, we show that functions in the homogeneous space $\dot N^{1,p}(X)$ necessarily have limits at infinity outside of $p$-thin sets, when $1\le p<Q<+\infty$.  When $p>Q$, we show by example that uniqueness of limits at infinity may fail for functions in $\dot N^{1,p}(X)$. While functions in $\dot N^{1,p}(X)$ may not have any reasonable limit at infinity when $p=Q$, we introduce the notion of a $Q$-thick set at infinity, and characterize the limits of functions in $\dot N^{1,Q}(X)$ along infinite curves in terms of limits outside $Q$-thin sets and along $Q$-thick sets.  By weakening the notion of a thick set, we show that a function in $\dot N^{1,Q}(X)$ with a limit along such an almost thick set may fail to have a limit along any infinite curve.  While homogeneous $p$-Sobolev functions may have infinite limits at infinity when $p\ge Q$, we provide bounds on how quickly such functions may grow: when $p=Q$, functions in $\dot N^{1,p}(X)$ have sub-logarithmic growth at infinity, whereas when $p>Q$, such functions have growth at infinity controlled by $d(\cdot, O)^{1-Q/p}$, where $O$ is a fixed base point in $X$.  For the inhomogeneous spaces $N^{1,p}(X)$, the phenomenon is different.  We show that for $1\le p\le Q$, the limit of a function $u\in N^{1,p}(X)$ is zero outside of a $p$-thin set, whereas $\lim_{x\to+\infty}u(x)=0$ for all $u\in N^{1,p}(X)$ when $p>Q$. 

\end{abstract}

 \section{Introduction}

 The aim of this paper is to understand the large-scale behaviour of homogeneous and inhomogeneous $p$-Sobolev functions, with $1\le p<+\infty$, when $(X,d,\mu)$ is an Ahlfors $Q$-regular metric measure space, with $1<Q<+\infty$, which supports a $p$-Poincar\'e inequality, as defined in Section~\ref{admissible}.  In particular, we explore whether the limit 
 \begin{equation}\label{eq1-2805}
		\lim_{x\to{+\infty}} u(x)
	\end{equation}
 exists for such Sobolev functions $u$.  Here  $x\to{+\infty}$ means that $|x|\to{+\infty}$ where $|x|:=d(O,x)$ for a given $O\in X$. 
 For $1\le p<+\infty$, the homogeneous Sobolev space, denoted $\dot N^{1,p}(X):=\dot N^{1,p}(X,d,\mu)$, is the space of locally integrable  functions that have a $p$-integrable  upper gradient on $X$, where the notion of upper gradients is given in Section \ref{modulus-capacity}.  As usual, the Sobolev space  $N^{1,p}(X):=N^{1,p}(X,d,\mu),$ $1\le p<+\infty$, is the collection of all $p$-integrable functions on $(X,d,\mu)$ belonging to $\dot  N^{1,p}(X)$.
 
 Recently, the second and third authors established the existence and uniqueness of limits at infinity of functions in $\dot N^{1,p}(X)$ for $1\le p<+\infty$, along infinite curves outside a zero modulus family.  This is under the assumption that the  space is equipped with a doubling measure which supports a $p$-Poincar\'e inequality, and that the space satisfies the annular chain property, see \cite{KN22} and Section~\ref{sec-dyadic}. 
 In particular, it was shown that for every $u\in \dot N^{1,p}(X),$ there exists a constant $c\in\mathbb R$, depending only on $u$, such that 
 	\begin{equation}\label{eq2-2805}
	\lim_{t\to{+\infty}}u(\gamma(t)) \text{\rm \ \ exists and is equal to $c$ for $p$-almost every $\gamma\in\Gamma^{+\infty}$}.
	\end{equation}
Here, $\Gamma^{+\infty}$ is the collection of all infinite curves, defined in Section \ref{sec-polar}. By $p$-almost every curve, we mean that the conclusion holds for every curve outside a family with zero $p$-modulus, see Section \ref{modulus-capacity}. 
If there exists $c\in[-\infty,+\infty]$ such that $\lim_{t\to+\infty}u(\gamma(t))=c$, then we say that $u$ has a limit along $\gamma$.  If $c\in\R$, then we say that $u$ has a finite limit along $\gamma$. For a recent thorough investigation of when $\dot N^{1,p}(X)=N^{1,p}(X)+\mathbb{R}$ in metric measure spaces of uniformly locally controlled geometry, see \cite{GKS24}. For each $1\leq p <\infty$, there exists a function $f\in \dot N^{1,p}(X)$ 
{ such that, although}
 the limit at infinity of $f$ along $p$-almost every infinite curve exists and equals to a finite constant by \eqref{eq2-2805}, the function $f$ may fail to be in $N^{1,p}(X)+\mathbb R$ by \cite[Theorem 1.1]{GKS24}.
Moreover, we refer the readers to  \cite{LN24} for a discussion on \eqref{eq2-2805} for bounded variation functions, and that the conditions of doubling and Poincar\'e inequalities are{
necessary} for the existence of limits at infinity along $1$-almost every infinite curve of bounded variation functions.

Under the assumptions of Ahlfors $Q$-regularity, we first show that we can improve{
the above result of \cite{KN22}} when $1\le p<\infty$ is strictly less than $Q$ and $X$ supports a $p$-Poincar\'e inequality.  For $1\le p\le Q$, we define the notion of a $p$-thin set, which roughly speaking has sufficiently small Hausdorff content in dyadic annuli at infinity. See Section~\ref{sec-thin} for the precise definitions.  We then obtain the following result:  
\begin{theorem}\label{thm:p<Q Thin}
    Let $1< Q<+\infty$ and let $1\le p<Q$.  Suppose that $(X,d,\mu)$ is a complete, unbounded metric measure space with metric $d$ and Ahlfors $Q$-regular measure $\mu$ supporting a $p$-Poincar\'e inequality. Then for every $u\in\dot N^{1,p}(X)$, there exists a constant $c\in\R$ and a $p$-thin set $E\subset X$ such that 
    \[
    \lim_{E\not\ni x\to+\infty}u(x)=c.
    \]
    Furthermore, there is only one such $c\in\R$ for which the above holds.
\end{theorem}

To prove this theorem, we use the result from \cite{KN22} which states that $u$ has a limit at infinity along a family of infinite curves with positive $p$-modulus.  By using this family of curves, we can construct the stated $p$-thin set.  However, when $p\ge Q$, the $p$-modulus of $\Gamma^{+\infty}$ is zero, see for instance \cite[Chapter 5]{pekka}, and so we cannot use this line of reasoning to guarantee the existence of limits of functions in $\dot N^{1,p}(X)$ along infinite curves in this case.  In fact, we show in Example~\ref{example-homogeneous} that there are functions in $\dot N^{1,p}(X)$, for both $p=Q$ and $p>Q$, which fail to have a limit at infinity along any curve $\gamma\in\Gamma^{+\infty}$.
For this reason, we consider the following question for the case $p\ge Q$: if a function in $\dot N^{1,p}(X)$ has a limit along some infinite curve, must it have limits along many such curves, and are these limits unique?

When $p>Q$, we see that uniqueness is not guaranteed, even when limits exist along a large collection of infinite curves.  Indeed, in Example~\ref{example-p>n}, we construct a function in $\dot W^{1,p}(\R^n)$ with $p>n$ which has a limit at infinity along every radial curve, but these limits do not coincide.  For this reason, the critical case $p=Q$ is one of the primary focuses of this paper.  It turns out that when $p=Q$, the limits at infinity of Sobolev functions are more well-behaved: in Lemma~\ref{thm1.2-0505}, we show that if a function $u\in\dot N^{1,Q}(X)$ has a limit along some infinite curve, then it obtains the same limit at infinity everywhere outside of a $Q$-thin set. 
Furthermore, we show that a function in $\dot N^{1,Q}(X)$ cannot have more than one such limit at infinity, in contrast to the case when $p>Q$.

A natural question is whether the existence of limits at infinity outside of a $Q$-thin set implies the existence of limits along infinite curves.  To answer this question, we introduce the notion of a {\it thick} set, defined in Section~\ref{sec-thin}.  Roughly speaking, thick sets have Hausdorff content uniformly bounded away from zero in sufficiently large dyadic annuli.  One of the main ideas of this paper, shown in  Lemmas~\ref{lem5.1-2308}, \ref{lem5.2-2308}, and \ref{lem5.3-0709}, is that if $u\in\dot N^{1,Q}(X)$ has a limit at infinity along some thick set, then there exists a large family of infinite curves along which $u$ attains the same limit.  We recall that any family of infinite curves in this setting will have $Q$-modulus zero, and so $Q$-modulus gives us no information on the size of this family.  However, this family of curves has large $Q$-modulus in each of the considered annuli.  As such, we show that inside any sufficiently large ball, the curves in this family have length proportional to the radius of the ball, and the union of their trajectories inside the ball has measure proportional to the measure of the ball.  That is to say, this family of curves fills up a significant portion of each large ball, but each curve is not too long. In Lemma~\ref{lem4.5-0703}, we show that if $E\subset X$ is a $Q$-thin set, then $X\setminus E$ is a thick set.  This completes the picture for the following main result:

\begin{theorem}\label{thm1.3-2208}
Let  $1<Q<+\infty$. Suppose that $(X,d,\mu)$ is a complete,  unbounded metric measure space with metric $d$ and Ahlfors $Q$-regular measure $\mu$ supporting a $Q$-Poincar\'e inequality. Let $c\in[-\infty,\infty]$. 

Then for every function $u\in\dot N^{1,Q}(X)$, the following are equivalent:
\begin{enumerate}
 \item[{\rm I.}] There exists an infinite curve $\gamma\in\Gamma^{+\infty}$ such that $\lim_{t\to+\infty}u(\gamma(t))=c$.
 \item[{\rm II.}] There is a thin set $E$ such that $\lim_{E\not\ni x\to+\infty}u(x)=c$.
 \item[{\rm III.}]   There is a thick set $F$ such that $\lim_{F\ni x\to+\infty}u(x)=c$.
 \item[{\rm IV.}] There is a family $\Gamma$ of infinite curves such that $\lim_{t\to+\infty}u(\gamma(t))=c$  for each $\gamma\in\Gamma$, satisfying the following for sufficiently large $R>0$: 
\begin{enumerate}
    \item [{\rm 1.}] $ \mu(B(O,R)\cap|\Gamma|)\simeq \mu(A_R\cap |\Gamma|)\simeq\mu(B(O,R))$,
     \item [{\rm 2.}] $\ell(\gamma\cap B(O,R))\simeq \ell(\gamma\cap A_R)\simeq R$ for each $\gamma\in\Gamma$,
\end{enumerate}
where $A_R:=B(O,R)\setminus B(O,R/2)$. 
\end{enumerate}
 Furthermore, there is at most one $c\in[-\infty,\infty]$ such that any of the above hold.  Here, $|\Gamma|$ denotes the union of the trajectories of the curves in the family $\Gamma$.
\end{theorem}
  Notice that we allow for the possibility that $c\in\{-\infty,\infty\}$, since functions in $\dot N^{1,Q}(X)$ may have infinite limits at infinity.  For example, consider the function $u(x)=\log(\log(e^e+|x|))$, which belongs to $\dot W^{1,n}(\R^n)$ when $n\ge 2$.

 It is unclear to what extent we can weaken the definition of a thick set and still obtain limits along infinite curves in the above result.  We define the notion of an {\it almost thick} set in Section~\ref{sec-thin} by replacing the uniform lower bound $\delta$ in Definition~\ref{def:ThickSets}, with an annulus-dependent lower bound $\delta_j$ in Definition \ref{def:almost-thick}.
 However in Example~\ref{ex:EasyAlmostThick}, we show that uniqueness of limits along almost thick sets fails in every complete Ahlfors $Q$-regular metric measure space supporting a $Q$-Poincar\'e inequality.  Furthermore, in Example~\ref{example5.4-0709}, we show that given any complete Ahlfors $Q$-regular measure metric space supporting a $Q$-Poincar\'e inequality, there exists an almost thick set $F$ and a function $u\in\dot N^{1,Q}(X)$ such that $u$ has a limit along $F$ but has no limit along any infinite curve.  Hence, we cannot replace the thick set $F$ in the above result with an almost thick set.  In both of these examples, however, the annulus-dependent lower bound $\delta_j$ from Definition~\ref{almost-thick} decays very rapidly.  We pose the following open question:
 \begin{question}
Is there a bound on the rate at which $\delta_j\to 0$, as $j\to+\infty$, so that if $u\in\dot N^{1,Q}(X)$ has a limit at infinity along an almost thick set with this decay rate, then $u$ also attains the same limit along some infinite curve?  
 \end{question}

Theorem~\ref{thm1.3-2208} shows that if a function $u\in\dot N^{1,Q}(X)$ has a limit along some infinite curve, then it has a limit along a family of curves which are not too long in any large ball, and such that the family of curves occupies a lot of volume.  While this family of curves is nice in this sense, we 
do not have a way to parameterize such a family, and we do not know how curves in this family may overlap. However, we would like to understand the behaviour at infinity of $\dot N^{1,Q}(X)$ functions under the assumption that the space possesses some preferred parameterized family of infinite curves.  To do so, we now assume that our space satisfies the  condition of a {\it weak polar coordinate system}, as introduced in \cite{KN22}.{
This system, defined in Section \ref{polar}, consists of \emph{a family of infinite curves}
\[
\Gamma^{\mathcal O}(\mathbb S)=\{\gamma^{\mathcal O}_\xi\in\Gamma^{+\infty}: \gamma_\xi^{\mathcal O}(0)=\mathcal O, \xi\in\mathbb S\}
\]
starting from a fixed point $\mathcal{O}$ indexed by an \emph{abstract set}{
$\mathbb{S}$ of indices}, where the parameter set $\mathbb{S}$ is equipped with a Radon probability measure $\sigma$ and a metric $d_{\mathbb S}$ for which the ``polar coordinate'' integration of every integrable function is controlled by integration with respect to $\mu$. 
We refer to curves in this family as {\it radial curves}. For $\mathbb A\subset\mathbb S$, we denote by $\Gamma^{\mathcal O}(\mathbb A)$ the subfamily of all infinite curves $\gamma_\xi^{\mathcal O}$ starting from $\mathcal O$ with respect to $\xi\in\mathbb A$.
 In \cite{KN22,EKN22}, it was shown that functions in $\dot N^{1,p}(X)$ have limits along radial curves when $p$ is strictly smaller than the doubling dimension of our space.  Here we are considering the critical case when  $X$ is Ahlfors $Q$-regular and $p=Q$, and so $p$ equals the doubling dimension of our space. 

A collection of infinite curves is said to be  \emph{pairwise disjoint (at infinity)}  if there is a ball $B$ such that any two distinct curves in the collection are disjoint outside $B$.
A subfamily $\Gamma^{\mathcal O}(\mathbb S')$  of pairwise disjoint infinite curves at infinity with $\sigma(\mathbb S')>0$, where $\mathbb S'\subset\mathbb S$, is said to admit a \emph{dyadic Lipschitz projection (at infinity)}  if  there are finite constants $C>0, r_0>0$ such that   there is a  mapping $p: (\Gamma^{\mathcal O}(\mathbb S'), d)\to (\mathbb S', d_{\mathbb S})$, from $x\in\gamma_\xi^{\mathcal O}\in \Gamma^{\mathcal O}(\mathbb S')$ to the corresponding parameter $\xi\in\mathbb S'$,  being  $C/r$-Lipschitz on each  $\Gamma^{\mathcal O}(\mathbb S')\cap B(\mathcal O,r)\setminus B(\mathcal O,r/2)$ with $r>r_0$. 

}
{
Under the assumptions that $X$ has a subfamily admitting a dyadic Lipschitz projection,}
we obtain the following improvement of Theorem~\ref{thm1.3-2208}, characterizing limits of $Q$-Sobolev functions along such a preferred family of curves.  In what follows, we denote the absolute continuity of $\sigma$ with respect to $\Ha^\alpha_{d_{\mathbb S}}$ by $\sigma\ll\Ha^\alpha_{d_\mathbb S}${
 where $\mathcal H^\alpha_{d_{\mathbb S}}$ is the $\alpha$-Hausdorff measure defined on the metric space $(\mathbb S, d_{\mathbb S})$}.

\begin{theorem}\label{thm4-2705}

{
Under the assumptions of Theorem \ref{thm1.3-2208}, we additionally assume that $X$ has a weak polar coordinate system $(\Gamma^{\mathcal O}(\mathbb S), \sigma, d_{\mathbb S})$ such that it has a subfamily $(\Gamma^{\mathcal O}(\mathbb S'), \sigma, d_{\mathbb S})$, where $\mathbb S'\subset\mathbb S$, of pairwise disjoint infinite curves with $\sigma(\mathbb S')>0$ admitting  a dyadic Lipschitz projection. }
Let $c\in[-\infty,\infty]$. Suppose that there exists some $\alpha_0\in (0,Q)$ such that{
$\sigma\ll\Ha^{\alpha_0}_{d_{\mathbb S}}$}.  Then for every function $u\in \dot N^{1,Q}(X)$, the following are equivalent:
\begin{enumerate}
    \item[1.] There exists an infinite curve $\gamma\in\Gamma^{+\infty}$ such that $\lim_{t\to{+\infty}}u(\gamma(t))=c.$
    \item[2.]  There exists a thin set $E$  such that  $\lim_{E\not\ni x\to{+\infty}}u(x)=c.$
    \item[3.]  There exists a thick set $F$  such that 
	$\lim_{F\ni x\to+\infty}u(x)=c.$
    \item[4.] For all $0<\alpha<Q$, we have that $\lim_{t\to{+\infty}} u(\gamma_\xi^{\mathcal O}(t))=c$ 
			for{
			$\mathcal H^{\alpha}_{d_{\mathbb S}}$}-a.e.{
			$\xi\in\mathbb S'$}. 
	\item[5.]$\lim_{t\to{+\infty}} u(\gamma_\xi^{\mathcal O}(t))=c$ for {
	$\Ha^{\alpha_0}_{d_{\mathbb S}}
$}-a.e.{
$\xi\in\mathbb S'$}.		
    \item[6.]  $\lim_{t\to{+\infty}} u(\gamma_\xi^{\mathcal O}(t))=c$ for $\sigma
$-a.e. {
$\xi\in\mathbb S'$}. 
\end{enumerate}
Furthermore, there is at most one $c\in[-\infty,\infty]$ such that any of the above hold.
\end{theorem}

 

{
If $\mathbb{S}\subset X$,} the Hausdorff measure on $\mathbb{S}$ is defined with respect to the metric $d$ inherited from $X$.  The assumption that there exists $\alpha_0\in (0,Q)$ for which $\sigma\ll\Ha^{\alpha_0}$ prevents singletons from having positive $\sigma$-measure.  We also note that not all $\alpha\in (0,Q)$ satisfy the condition $\sigma\ll\Ha^{\alpha}$.  In $\R^n$, for example, the natural polar coordinate system is such that $\sigma\simeq\Ha^{n-1}$, and so $\alpha\in (n-1,n)$ will fail this condition. 

Recall that functions in $\dot N^{1,Q}(X)$ may have infinite limits at infinity, as mentioned above. 
 However, in the following result, we show that functions in $\dot N^{1,Q}(X)$ have sub-logarithmic growth at infinity outside a $Q$-thin set under the assumption of Theorem \ref{thm1.3-2208}.  Thus, we obtain a bound on how quickly such functions may approach infinity.
 
 \begin{theorem}\label{thm1-2705}
 Let $1<Q<{+\infty}$. 
Suppose that $(X,d,\mu)$ is a complete, unbounded metric measure space with metric $d$ and $Q$-Ahlfors regular measure $\mu$ supporting a $Q$-Poincar\'e inequality. Then for every function $u\in \dot N^{1,Q}(X),$ there exists a $Q$-thin set $E$ such that
 	\begin{equation}
		\label{eq1.1-2705}
		\lim_{E\not\ni x\to{+\infty}} \left | \frac{u(x)}{\log^{\frac{Q-1}{Q}}(|x|)} \right|=0.
	\end{equation}

Moreover, we have that 
	\begin{equation}\label{equ1.2-2805}
		\lim_{x\to{+\infty}} \left | \frac{u_{B(x,1)}}{\log^{\frac{Q-1}{Q}}(|x|)}\right|=0 \text{\rm \ \ for every function $u\in \dot  N^{1,Q}(X)$ }
	\end{equation}
where
 $u_{B(x,1)}:=\dashint_{B(x,1)}ud\mu:=\frac{1}{\mu(B(x,1))}\int_{B(x,1)}ud\mu$.
 \end{theorem}

When $p>Q$, we modify the arguments used in the proof of the above theorem to obtain the following improved growth bounds, under the additional assumption of the annular chain property, see Section~\ref{sec-dyadic}:
\begin{theorem}\label{thm:p>Q Growth}
Let $1<Q<+\infty$, and let $Q<p<+\infty$.  Suppose that $(X,d,\mu)$ is a complete, unbounded metric measure space with metric $d$ and Ahlfors $Q$-regular measure $\mu$ supporting a $p$-Poincar\'e inequality, and suppose that $X$ satisfies the annular chain property, as in Definition~\ref{dyadic}.  Then for every function $u\in\dot N^{1,p}(X)$, we have that 
\begin{equation}\label{eq:p>Q Growth 1}
\lim_{x\to+\infty}\frac{|u(x)|}{|x|^{{1-Q/p}}}=0,
\end{equation}
and also that 
\begin{equation}\label{eq:p>Q Growth 2}
\lim_{x\to+\infty}\frac{|u_{B(x,1)}|}{|x|^{1-Q/p}}=0.
\end{equation}  
\end{theorem}
Notice that when $p>Q$, the annular chain property is not a consequence of the assumptions of completeness, Ahlfors $Q$-regularity, and $p$-Poincar\'e inequality, as in the case $p\le Q$ \cite[Theorem~3.3]{K07}.  Because of this, we include it as an additional assumption in the statement of the above theorem.

Until this point, our results have concerned only the homogeneous space $\dot N^{1,p}(X)$.  We conclude by turning our attention to the inhomogeneous space $N^{1,p}(X)$.  When $1\le p<Q$, we use the results from \cite{KN22} to show that if $u\in N^{1,p}(X)$, then $\lim_{t\to+\infty}u(\gamma(t))=0$ for $p$-a.e $\gamma\in\Gamma^{+\infty}$.
By Theorem~\ref{thm:p<Q Thin} above, it then follows that $u$ attains the same limit at infinity outside of a $p$-thin set.  When $p>Q$ and $X$ is complete, we show in Proposition~\ref{thm2.2-07-03} that $\lim_{x\to+\infty} u(x)=0$ for every  $u\in N^{1,p}(X)$.  When $p=Q$, this is not the case: there exists a function $f\in N^{1,Q}(X)$ such that $\lim_{x\to+\infty}f(x)$ fails to exist. For example, consider the function $\sum_{j\in\mathbb N}\psi_j\in N^{1,Q}(X)$ constructed Example~\ref{example5.4-0709}.  However, in the case $p=Q$, we similarly show that functions in $N^{1,Q}(X)$ have limit zero outside of a $Q$-thin set.  We summarize these results with the following theorem:

\begin{theorem}\label{thm1.1-3105}
Let $1<Q<+\infty$ and $1\le p<+\infty$.   Suppose that $(X,d,\mu)$ is a complete unbounded metric measure space with metric $d$ and Ahlfors $Q$-regular measure $\mu$ supporting a $p$-Poincar\'e inequality.  If $Q<p<+\infty$, then for every $u\in N^{1,p}(X)$, we have that
\begin{equation}\label{eq:p>Q Inhom}
\lim_{x\to+\infty} u(x)=0.
\end{equation}
If $1\le p\le Q$, then for every $u\in N^{1,p}(X)$, there exists a $p$-thin set $E$ such that 
	\begin{equation}\label{eq1.4-0506}
		\lim_{E\not\ni x \to+\infty}u(x)=0.
	\end{equation}
 Furthermore, if there exists $c'\in\R$ and a $p$-thin set $E'\subset X$, such that $\lim_{E'\not\ni x\to+\infty}u(x)=c'$, then we have that $c'=0$. 
\end{theorem}

We note that the conclusion of Theorem~\ref{thm1.1-3105} does not hold for functions in the homogeneous space $\dot N^{1,p}(X)$ when $p\ge Q$, as shown in Example~\ref{example-homogeneous}.  

The outline of the paper is as follows: 
 \tableofcontents

Throughout this paper, we use the following conventions. We denote by $O$ the base point in the annular chain property and denote by $\mathcal O$ the coordinate 
point in the weak polar coordinate system \eqref{polar-0103}. These points do not need to be equal, but this does not cause any logical problems in what follows. 
The notation $A\lesssim B\ (A\gtrsim B)$ means that there is a constant $C>0$, depending only on the structural constants associated with the Ahlfors $Q$-regularity assumption or $p$-Poincar\'e inequality for example,  such that $A\leq C \cdot B\ (A\geq C\cdot B)$.  Likewise, by $A\approx B$ (or $A\simeq B$), we mean that $A\lesssim B$ and $A\gtrsim B$. 

For  each locally integrable function $f$ and for every measurable subset $A\subset X$ of strictly positive measure, we let $f_A:=\dashint_Afd\mu=\frac{1}{\mu(A)}\int_Afd\mu$.  We denote the ball centered at $x\in X$ of radius $r>0$ by $B(x,r)$, and for $\tau>0$, we set $\tau B(x,r):=B(x,\tau r)$.

\section{Preliminaries}
\subsection{Infinite curves}\label{sec-polar}\ 

Let $(X,d)$ be a metric space. A  \textit{curve} is a nonconstant continuous mapping from an interval $I\subseteq\mathbb R$ into $X$. The \textit{length} of a curve $\gamma$ is denoted by $\ell(\gamma)$. A curve $\gamma$ is said to be a  \textit{rectifiable curve} if its length is finite. Similarly, $\gamma$ is a   \textit{locally rectifiable curve} if its restriction to each compact subinterval of $I$ is rectifiable. Each rectifiable curve $\gamma$ will be parameterized by arc length and hence the  \textit{line integral} over $\gamma$ of a Borel function $f$ on $X$ is 
\[\int_\gamma fds =\int_0^{\ell(\gamma)}f(\gamma(t))dt.
\]
If $\gamma$ is locally rectifiable, then we set 
\[\int_{\gamma}fds=\sup\int_{\gamma'}fds
\]where the supremum is taken over all rectifiable subcurves $\gamma'$ of $\gamma$. Let $\gamma:[0,{+\infty})\to X$ be a locally rectifiable curve, parameterized by arc length. Then
\[\int_{\gamma}fds=\int_0^{+\infty} f(\gamma(t))dt.
\]
A locally rectifiable curve $\gamma$ is an  \textit{infinite curve} if $\gamma\setminus B\neq\emptyset$ for all balls $B$. Then $\int_{\gamma}ds={+\infty}.$   
 We denote by $\Gamma^{+\infty}$ the collection of all infinite curves.

\subsection{Modulus and capacities}\label{modulus-capacity}
\ 

A triple $(X,d,\mu)$ is called a metric measure space if $(X,d)$ is a separable metric space and $\mu$ is a nontrivial locally finite Borel-regular measure on $X$. For every set $A$, we set 
 \[
 \mu(A):=\inf\{\mu(U): A\subset U, U\ \text{\rm is a $\mu$-measurable subset of $X$} \}.
 \]
 By \cite[Proposition 3.3.37]{pekka}, the above $\mu$-measurable subset $U$ can be replaced by an open subset.

 Let $\Gamma$ be a family of curves in a metric measure space $(X,d,\mu)$. Let $1\le p<+\infty$. The  \textit{$p$-modulus} of $\Gamma$, denoted  $\text{\rm Mod}_p(\Gamma)$, is defined by 
\[\text{\rm Mod}_p(\Gamma):=\inf \int_{X}\rho^pd\mu
\]where the infimum is taken over all Borel functions $\rho:X\to[0,{+\infty}]$ satisfying 
$\int_{\gamma}\rho ds\geq 1
$ for every locally rectifiable curve $\gamma\in\Gamma$. A family of curves is called \textit{$p$-exceptional} if it has $p$-modulus zero. We say that a property  holds for \textit{$p$-a.e curve} (or $p$-almost every curve) if the collection of curves for which the property fails  is $p$-exceptional. 

Let $E,F$ be two subsets of a set $\Omega$. We define 
\begin{equation}\label{eq2.1-0506}
\text{\rm Mod}_p(E,F,\Omega):= \inf\int_\Omega\rho^pd\mu
\end{equation}
where the infimum  is taken over all Borel functions $\rho:\Omega\to[0,+\infty]$ satisfying $\int_{\gamma}\rho ds\geq 1$ for every rectifiable curve $\gamma$ in $\Omega$ connecting $E$ and $F$. If $\Omega=X$, we set $\text{\rm Mod}_p(E,F):=\text{\rm Mod}_p(E,F,X)$.

Let $u$ be a locally integrable function on $X$. A Borel function $\rho:X\to[0,{+\infty}]$ is said to be an \textit{upper gradient} of $u$ if 
\begin{equation}
\label{def-upper-gradient}|u(x)-u(y)|\leq \int_{\gamma}\rho\, ds
\end{equation}
for every rectifiable curve $\gamma$ connecting $x$ and $y$. Then we  have that \eqref{def-upper-gradient} holds for all compact subcurves of $\gamma\in\Gamma^{+\infty}$. For $1\le p<+\infty$, we say that $\rho$ is a  \textit{$p$-weak upper gradient} of $u$ if \eqref{def-upper-gradient} holds for $p$-a.e rectifiable curve. In what follows, we denote by $\rho_u$ the  \textit{minimal $p$-weak} upper gradient of $u$, which is unique up to sets of measure zero and which is minimal in the sense that $\rho_u\leq\rho $ a.e.\ for every $p$-integrable $p$-weak upper gradient $\rho$ of $u$. In \cite{H03}, the existence and uniqueness of such  minimal $p$-weak upper gradient are given. 
The notion of upper gradients is due to Heinonen and Koskela \cite{HK98}, and we refer  interested  readers to \cite{BB15,H03,HK98,N00} for a more detailed discussion on upper gradients.

Let $\Omega\subseteq X$. We denote $\dot N^{1,p}(\Omega):=\dot N^{1,p}(\Omega,d,\mu)$, with $1\leq p<+\infty$, the collection of all locally integrable functions for which an upper gradient is $p$-integrable on $(\Omega,d,\mu)$. We set $N^{1,p}(\Omega):=\dot N^{1,p}(\Omega)\cap L^p_\mu(\Omega)$ where $1\leq p<+\infty$ and $L^p_\mu(\Omega)$ consists of all $p$-integrable functions on $(\Omega,\mu)$.

Let $K\subset X$ be a subset. The \emph{Sobolev $p$-capacity} of $K$, denoted $C_p(K)$ is defined by 
\begin{equation}\label{eq2.3-17Dec}
C_p(K):=\inf\left(\int_X |u|^pd\mu+\int_X\rho_u^pd\mu\right),
\end{equation}
where the infimum is taken over all functions $u\in N^{1,p}(X)$ with the $p$-integrable minimal $p$-weak upper gradient $\rho_u$ such that $u|_K\equiv 1$.

Given two subsets $E,F$ of a set $\Omega$, we set 
\begin{equation}\label{eq2.3-0506}
	\text{\rm Cap}_p(E,F,\Omega):=\inf\int_\Omega \rho_u^pd\mu
\end{equation}
where the infimum is taken over all functions $u\in\dot N^{1,p}(\Omega)$ with the $p$-integrable minimal $p$-weak upper gradient $\rho_u$ such that $u|_E\equiv 1$ and $u|_{F}\equiv 0$. If $\Omega=X,$ we set $\text{\rm Cap}_p(E,F):=\text{\rm Cap}_p(E,F,\Omega)$.  

{By \cite[Lemma 6.2.2]{pekka}, one can replace the minimal $p$-weak upper gradients $\rho_u$ in \eqref{eq2.3-17Dec}-\eqref{eq2.3-0506} with upper gradients of the function $u$.}   


We have from \cite[Proposition 2.17]{HK98} that for two subsets $E,F$ of a set $\Omega$,
\begin{equation}\label{eq2.4-2308}
\text{\rm Mod}_p(E,F,\Omega)=\text{\rm Cap}_p(E,F,\Omega).
\end{equation}
\subsection{Ahlfors regular measures  and Poincar\'e inequalities}\label{admissible} \ 

Let $(X,d)$ be a metric space. A Borel regular measure $\mu$ is called \textit{doubling} if every ball in $X$ has  finite positive  measure and if there exists a finite constant $C\geq 1$ such that for all balls $B(x,r)$  with radius $r>0$ and center at $x\in X$,
\[0<\mu(B(x,2r))\leq C\mu(B(x,r))<+\infty.
\]

Let $1< Q<{+\infty}$. A Borel regular measure $\mu$ is said to be \textit{Ahlfors $Q$-regular} if there exists a finite constant $C_Q\geq 1$ such that for all balls $B(x,r)$ with radius $r>0$ and center at $x\in X$,
\[\frac{r^Q}{C_Q}\leq \mu(B(x,r))\leq C_Qr^Q.
\] Hence if $\mu$ is Ahlfors $Q$-regular where $1<Q<{+\infty}$, then $\mu$ is a doubling measure. Here $C_Q$ is said to be the {\it Ahlfors $Q$-regularity constant} of $\mu$.

Let $1\leq p<{+\infty}$. We say that a measure $\mu$ supports a  \textit{$p$-Poincar\'e inequality} if every ball in $X$ has  finite positive measure and if there exist finite constants $C>0$ and $\lambda\geq 1$ such that 
\[\dashint_{B(x,r)}|u-u_{B(x,r)}|d\mu \leq C r\left (\dashint_{B(x,\lambda r)}\rho^pd\mu\right )^{1/p}
\] for all balls $B(x,r)$ with radius $r>0$ and center at $x\in X$, and for all pairs $(u,\rho)$ satisfying \eqref{def-upper-gradient} such that $u$ is integrable on balls. Here $\lambda$ is called the {\it scaling constant} or {\it scaling factor} of the $p$-Poincar\'e inequality. 

We will use the following Morrey embedding theorem, see \cite[Theorem~9.2.14]{pekka}:
\begin{theorem}\label{thm:Morrey}
Let $1<Q<+\infty$ and let $Q<p<+\infty$.  Suppose that $(X,d,\mu)$ is a complete, unbounded,  Ahlfors $Q$-regular metric measure space supporting a $p$-Poincar\'e inequality.  Then there exists a finite constant $C\ge 1$ such that given a ball $B\subset X$ and a function $u\in\dot N^{1,p}(X)$, we have that 
\[
|u(x)-u(y)|\le C\rad(B)^{Q/p}d(x,y)^{1-Q/p}\left(\dashint_{4\lambda B}\rho_u^pd\mu\right)^{1/p}
\]
for all $x,y\in B$.  Here $\lambda$ is the scaling factor of the Poincar\'e inequality, and $\rad(B)$ is the radius of $B$.
\end{theorem}

For more on Poincar\'e inequalities, we refer the interested readers to \cite{HK00,HK95}.
 
We will use the following $p$-modulus and $p$-capacity estimates for annular regions.  For proof, see for instance \cite[Proposition 2.17 and Lemma 3.14]{HK98} and \cite[Theorem 9.3.6]{pekka}.


\begin{theorem}\label{thm2.1-0506}
Let  $1<Q<+\infty$ and let $1\le p<+\infty$.
Suppose that $(X,d,\mu)$ is an unbounded metric measure space with metric $d$ and Ahlfors $Q$-regular measure $\mu$ supporting a $p$-Poincar\'e inequality.
We assume that $X$ is complete. 
Then for $0<2r<R<+\infty$, we have that
\[
	\text{\rm Mod}_p(B(O,r), X\setminus B(O,R)) \approx \text{\rm Cap}_p(B(O,r), X\setminus B(O,R))\simeq
 \begin{cases}
     r^{Q-p},&p<Q;\\
     \left(\log \left(\frac{R}{r}\right)\right)^{1-Q},& p=Q;\\
     R^{Q-p},& p>Q;
 \end{cases}  
\]
for a given $O\in X$.
\end{theorem}

\subsection{Chain conditions}
\label{sec-dyadic} \ 

In this paper, we employ the following  annular chain property which is given in \cite[Section 2.5]{KN22}.
\begin{definition} 
	\label{dyadic}
	Let $\lambda\geq1$. We say that $X$ satisfies an annular $\lambda$-chain condition at $O$  if the following holds. There are finite constants $c_1\geq 1, c_2\geq 1, \delta>0$ and a finite number $M<{+\infty}$  so that given $r>0$ and points $x,y\in B(O,r)\setminus B(O,r/2)$, one can find balls $B_1, B_2, \ldots, B_k$ with the following properties:
	\begin{enumerate}
		\item[1.] $k\leq M$.
		\item[2.] $B_1=B(x,r/(\lambda c_1))$, $B_k=B(y,r/(\lambda c_1))$ and the radius of  each $B_i$ is $r/(\lambda c_1)$ for $1\leq i\leq k$. 
		\item[3.] $ B_i\subset B(O,c_2r)\setminus B(O,r/c_2)$ for $1\leq i\leq k$.
		\item[4.] For each $1\leq i\leq k-1$, there is a ball $D_i\subset B_i\cap B_{i+1}$ with radius $\delta r$.
	\end{enumerate}
	If $X$ satisfies an annular $\lambda$-chain condition at $O$ for every $\lambda\geq 1$, we say that $X$ has the annular chain property.
\end{definition}

By \cite[Lemma 2.3]{KN22}, we have the following sufficient criterion guaranteeing the annular chain property holds:
\begin{proposition} \label{prop2.3}
Let $\mu$ be doubling on an unbounded metric space $(X,d)$.
Suppose that there is a finite constant $c_0\geq 1$ so that for every $r>0$, each pair $x,y$ of points in $B(O,r)\setminus B(O,r/2)$ can be joined by a curve in $B(O,c_0r)\setminus B(O,r/c_0)$. Then $X$ has the annular chain property.
\end{proposition}

By \cite[Theorem~3.3]{K07}, a complete, Ahlfors $Q$-regular metric measure space $(X,d,\mu)$ supporting a $p$-Poincar\'e inequality with $1\le p\le Q$ satisfies the hypotheses of Proposition~\ref{prop2.3}, and so such a metric measure space has the annular chain property.  Furthermore, from the proof of \cite[Lemma~2.3]{KN22}, we see that in such spaces, only the constants $M$ and $\delta$ in Definition~\ref{dyadic} depend on $\lambda$, whereas the constants $c_1$ and $c_2$  are independent of $\lambda$. In such spaces, we also have the following modification of the annular chain property: 
\begin{proposition}\label{prop2.3-0806} Suppose that $X$ is complete and $\mu$ is an Ahlfors $Q$-regular measure supporting a $p$-Poincar\'e inequality on an unbounded metric space $(X,d)$ where $1<Q<+\infty$ and $1\le p\le Q$. Then
there are finite constants $\overline{c_1}\geq 1, \overline {c_2}\geq 1, \delta>0$ and a finite number $2\leq \overline M<{+\infty}$  so that given $+\infty>r_2\geq 2r_1>0$ and two balls $B$ and $B'$ with radius $\frac{r_2}{r_1}$ in $B(O,\overline{c_1} r_2)\setminus B(O,r_1/\overline{c_1})$, one can find balls $B_1, B_2, \ldots, B_k$ with the following properties:
	\begin{enumerate}
		\item[1.] $k\leq \overline M$.
		\item[2.] $B_1=B$, $B_k=B'$ and the radius of  each $B_i$ is comparable to $r_2/r_1$ for $1\leq i\leq k$. 
		\item[3.] $ B_i\subset B(O,\overline{c_2}r_2)\setminus B(O,r_1/\overline{c_2})$ for $1\leq i\leq k$.
		\item[4.] For each $1\leq i\leq k-1$, there is a ball $D_i\subset B_i\cap B_{i+1}$ with radius $\delta \frac{r_2}{r_1}$.
	\end{enumerate}
\end{proposition}
\begin{proof} Since $X$ is complete and $\mu$ is an Ahlfors $Q$-regular measure supporting a $p$-Poincar\'e inequality where $1<Q<+\infty$ and $1\le p\le Q$, we have that the assumption of Proposition \ref{prop2.3} holds, see for instance in \cite[Theorem 3.3]{K07}. In particular, there is a finite constant $c_0\geq 1$ so that for every $+\infty>r_2\geq 2r_1>0$, each pair $x,y$ of points in $B(O,r_2)\setminus B(O,r_1)$ can be joined by a curve in $B(O,c_0r_2)\setminus B(O,r_1/c_0)$. Hence the claim follows by an analog of the proof of Proposition \ref{prop2.3} (see \cite[Lemma 2.3]{KN22}).
\end{proof}
%
\subsection{Loewner spaces}\ 

Let $(X,d,\mu)$ be a metric measure space. For $1<Q<+\infty$, we define the {\it Loewner function} $\phi_{X,Q}:(0,+\infty)\to [0,+\infty)$ of $X$ by 
\[
	\phi_{X,Q}(t)=\inf \{\text{\rm Mod}_Q(E,F,X): \Delta(E,F)\leq t \}
\]
where $E$ and $F$ are disjoint non-degenerate continua in $X$ with 
	\[
		\Delta(E,F)=\frac{\text{\rm dist}(E,F)}{\min\{\text{\rm diam}E, {\rm diam} F\}}
	\]
designating their relative distance in $X$. Here ${\rm dist}(E,F):=\inf\{d(x,y): x\in E, y\in F \}$ and ${\rm diam}E:=\sup\{d(x,y): x,y\in E\}$. If one can not find two disjoint continua in $X$, it is understood that $\phi_{X,Q}(t)\equiv 0$. By the definition, the function $\phi_{X,Q}$ is decreasing.

Let $1<Q<+\infty$, a pathwise connected metric measure space $(X,d,\mu)$ is said to be a {\it Loewner space} of exponent $Q$ or a {\it $Q$-Loewner space}, if the Loewner function $\phi_{X,Q}(t)$ is positive for all $t>0$. 

In fact, we have the following theorems.
\begin{theorem}\label{thm2.6-0709}
Let $1<Q<+\infty$.
 Let $(X,d,\mu)$ be a complete, unbounded metric measure space with metric $d$ and Ahlfors $Q$-regular measure $\mu$ supporting a $Q$-Poincar\'e inequality. 
Then $X$ is a $Q$-Loewner space.
\end{theorem}
\begin{proof}
By Theorem 9.6 in \cite{Hei01}, we have that if $(X,d,\mu)$ is quasiconvex, proper, Ahlfors $Q$-regular and supports a $Q$-Poincar\'e inequality then $X$ is a $Q$-Loewner space. Since $X$ is complete and Ahlfors $Q$-regular supports a $Q$-Poincar\'e inequality, we obtain that $X$ is proper (see for instance \cite[Lemma 4.1.14]{pekka}) and quasiconvex (see for instance \cite[Theorem 8.3.2]{pekka}). The proof completes.
\end{proof}
\begin{theorem} [Section 3 in \cite{HK98}] \label{thm2.7-0709}
 Let $1<Q<+\infty$.
 Let $(X,d,\mu)$ be a complete, unbounded metric measure space with metric $d$ and Ahlfors $Q$-regular measure $\mu$ supporting a $Q$-Poincar\'e inequality. 
Let $0<r<+\infty$, and let $E,F$ be disjoint non-degenerate continua which are subsets of $B(O,Cr)\setminus B(O,r/C)$, for a given $+\infty>C>1$, such that $\min\{\diam E,\diam F\}\ge r/8$. Then there is a finite constant $\delta>0$ independent of $r$ so that 
 \[
 	{\rm Mod}_Q(E,F, B(O,Cr)\setminus B(O,r/C)) \geq \delta.
 \]
\end{theorem}

\subsection{Polar coordinate systems}\label{polar} \ 

The following weak polar coordinate system was introduced in \cite{KN22}.
{
Let $\mathbb S$ be an abstract set of indices with a Radon probability measure $\sigma$ and a metric $d_{\mathbb S}$.}
Given a point $\mathcal O$, we  consider collections $\Gamma^{\mathcal O}(\mathbb S)$ of infinite curves with parameter set $\mathbb S$ starting from $\mathcal O$, namely
\[\Gamma^{\mathcal O}(\mathbb S)=\{\gamma^{\mathcal O}_\xi\in\Gamma^{+\infty}: \gamma_\xi^{\mathcal O}(0)=\mathcal O, \xi\in\mathbb S\}.
\]
We say that $X:=(X,d,\mu)$  has a   \textit{weak polar coordinate system} at the \textit{coordinate point} $\mathcal O$ if there is a choice of a pair $(\mathbb S, \Gamma^{\mathcal O}(\mathbb S))$  with a  Radon probability measure $\sigma$ on $\mathbb S$, a \textit{coordinate weight} $h:X\to\mathbb [0,{+\infty})$, and a finite constant $\mathcal C>0$ such that 
\begin{equation}\label{polar-0103}
\int_{\mathbb S}\int_{\gamma^{\mathcal O}_\xi}|f|\ h\ dsd\sigma\leq\mathcal C \int_{X}|f|d\mu \text{\rm \ \ for every integrable function $f$ on $(X,d,\mu)$.}
\end{equation}
Each infinite curve $\gamma_\xi^{\mathcal O}\in\Gamma^{\mathcal O}(\mathbb S)$ is called a   \textit{radial curve} with respect to $\xi\in\mathbb S$ (starting from $\mathcal O$). We refer the interested readers to \cite[Section 2.1]{KN22} for more discussion on general and concrete polar  coordinates.

By a slight abuse of notation, we will also denote by $\Gamma^{\mathcal O}(\mathbb S)$  the subset of  $X$ which is the union of all   radial curves. 
{
A collection of infinite curves is said to be  \emph{pairwise disjoint (at infinity)}  if there is a ball $B$ such that for any two distinct curves $\gamma, \gamma'$ in the collection, $(\gamma\cap\gamma' ) \setminus B=\emptyset$. Given a weak polar coordinate system $(\Gamma^{\mathcal O}(\mathbb S),\sigma, d_{\mathbb S}):=(\mathbb S,\Gamma^{\mathcal O}(\mathbb S), h, \sigma, d_{\mathbb S})$ at $\mathcal O$, 
a subfamily $\Gamma^{\mathcal O}(\mathbb S')$  of pairwise disjoint infinite curves with $\sigma(\mathbb S')>0$, where $\mathbb S'\subset\mathbb S$, is said to admit a \emph{dyadic Lipschitz projection (at infinity)}  if  there are finite constants $C>0, r_0>0$ such that   there is a  mapping $p: (\Gamma^{\mathcal O}(\mathbb S'), d)\to (\mathbb S', d_{\mathbb S})$, from $x\in\gamma_\xi^{\mathcal O}\in \Gamma^{\mathcal O}(\mathbb S')$ to the corresponding parameter $\xi\in\mathbb S'$,  being  $C/r$-Lipschitz on each  $\Gamma^{\mathcal O}(\mathbb S')\cap B(\mathcal O,r)\setminus B(\mathcal O,r/2)$ with $r>r_0$. 
For example, this hold for the Muckenhoupt-weighted space $\mathbb R^n$ where $1<n\in\mathbb N$.
}

  
\subsection{Hausdorff measures, thin sets, and thick sets}
\label{sec-thin}		\ 

Let $0\le \beta<{+\infty}$ and $0<R\leq{+\infty}$. {
Let $(X,d)$ be an arbitrary metric space. The {\it $(\beta,R)$-Hausdorff content}  of a subset $E$, denoted  $\mathcal H^\beta_{R, d}(E)$, is given by
			\begin{equation}
			\notag
			\mathcal H^\beta_{R,d}(E)=\inf\left\{\sum_{k\in\mathbb N} r_k^\beta: E\subset \bigcup_{k\in\mathbb N}B_k\text{\rm \ \ and \ \ } 0<r_k<R \right\}
			\end{equation}
	where $B_k$ is a ball with radius $r_k$. The {\it $\beta$-Hausdorff measure} $\mathcal H_d^\beta$ of a subset $E$ is 
			\[
				\mathcal H_{d}^{\beta}(E):=\lim_{R\to 0}\mathcal H^\beta_{R,d}(E).
			\]
For brevity, we will omit the metric $d$ when referring to  a given metric measure space $(X,d,\mu)$, and  denote by $\mathcal H^\beta_R:=\mathcal H^\beta_{R,d}$ and 	$\mathcal H^{\beta}:=\mathcal H_{d}^{\beta}$. In this paper, {
to emphasize the``spherical" metric space $(\mathbb S, d_{\mathbb S})$ defined in Section \ref{polar}, we denote the Hausdorff contents and Hausdorff measures defined with respect to this space} as $\mathcal H^\beta_{R, d_{\mathbb S}}$ and $\mathcal H^\beta_{d_{\mathbb S}}$, respectively.
 }
 
   If $\mu$ is Ahlfors $Q$-regular, $1<Q<+\infty$, then for arbitrary set $A\subset X$, $\mu(A)\simeq \mathcal H^Q(A)$.
   When the measure $\mu$ is only assumed to be doubling, and not necessarily Ahlfors $Q$-regular, it is more natural to define the  \emph{codimension $t$ Hausdorff content} of a subset $E$, for $0< t<+\infty$ and $0<R\le+\infty$, as follows:
   \[
   \Ha^{-t}_R(E)=\inf\left\{\sum_{k\in\N}\frac{\mu(B_k)}{r_k^t}:E\subset\bigcup_{k\in\N}B_k\quad\text{and}\quad 0<r_k<R\right\}.
   \]
   The \emph{codimension $t$ Hausdorff measure} of a subset $E$ is then defined by 
   \[
   \Ha^{-t}(E):=\lim_{R\to 0}\Ha^{-t}_R(E).
   \]
   If $\mu$ is  Ahflors $Q$-regular, then $\Ha^{-t}_R\simeq\Ha^{Q-t}_R$ and $\Ha^{-t}\simeq\Ha^{Q-t}$ for $0< t\le Q$.
   
	We have two following results, the first of which can be found in  \cite[Lemma 2.25]{HKM06}.  The second is obtained by connectedness of infinite curves, see \cite[Corollary 2.10.12]{Fe69}.
	\begin{lemma}\label{lem2.1-0805}
		Let $0<\beta<{+\infty}$, $0<R\leq {+\infty}$. Then for $E\subset X,$ $\mathcal H^{\beta}_{R}(E)=0$ if and only if  $\mathcal H^{\beta}(E)=0$.
	\end{lemma}
 
\begin{lemma}\label{lem2.3-0505}
 	Let $+\infty>R>0$. Then	
 		$\mathcal H^1_{+\infty}(\gamma\bigcap B(O,2R)\setminus B(O,R))\geq R/2$ for each $\gamma\in\Gamma^{+\infty}$ such that $\gamma\bigcap B(O,R)\ne\varnothing$.
 	\end{lemma}

  \begin{lemma}\label{lem:Measure-Content Bound}
  Let $1<Q<+\infty$ and let $1\le p\le Q$.  Suppose that $(X,d,\mu)$ is Ahlfors $Q$-regular, and for $+\infty>\lambda\ge 1$, we denote $A_{\lambda 2^{j+1}}:=B(O,\lambda 2^{j+2})\setminus B(O,\lambda 2^{j+1})$.  Then for $G\subset A_{\lambda 2^{j+1}}$, we have that 
  \[
  \frac{\mu(G)}{(\lambda 2^{j+1})^p}\lesssim\frac{\Ha^{Q-(p-\alpha)}_{+\infty}(G)}{(\lambda 2^{j+1})^\alpha}\quad\text{for all }0<\alpha<p.
  \]      
  \end{lemma}

  \begin{proof}
    Consider a cover $\{B(x_k,r_k)\}_{k\in\mathbb N}$ of $G$.  We have that
		\[
		\frac{\mu(G)}{(\lambda 2^{j+1})^p}\leq \frac{\sum_{k\in\mathbb N} \mu(B(x_k,r_k))}{(\lambda 2^{j+1})^p}  \simeq \frac{\sum_{k\in\mathbb N} r_k^Q}{(\lambda 2^{j+1})^p}  = \frac{1}{(\lambda 2^{j+1})^\alpha} \sum_{k\in\mathbb N} \frac{r_k^{p-\alpha}}{(\lambda 2^{j+1})^{p-\alpha}} r_k^{Q-(p-\alpha)}.
		\]
		If $0<r_k<\lambda 2^{j+1}$ for all $k\in\mathbb N,$ then $r_k^{p-\alpha}< (\lambda 2^{j+1})^{p-\alpha}$ for all $0<\alpha<p$. Hence 
		\[
		\frac{\mu(G)}{(\lambda 2^{j+1})^p} \lesssim \frac{1}{(\lambda 2^{j+1})^\alpha} \sum_{k\in\mathbb N} r_k^{Q-(p-\alpha)} \text{\rm \ \ for all } 0<\alpha<p.
		\]
		If there is an index $k_0$ such that  $r_{k_0}\ge \lambda 2^{j+1}$, then  the Ahlfors $Q$-regular properties give that
		\[
		\frac{\mu(G)}{(\lambda 2^{j+1})^p} \leq \frac{\mu(A_{\lambda 2^{j+1}})}{(\lambda 2^{j+1})^p}\simeq\frac{(\lambda 2^{j+1})^{Q-(p-\alpha)}}{(\lambda 2^{j+1})^\alpha} \le\frac{r_{k_0}^{Q-(p-\alpha)}}{(\lambda 2^{j+1})^\alpha}\leq\frac{1}{ (\lambda 2^{j+1})^\alpha}\sum_{k\in\mathbb N}r_k^{Q-(p-\alpha)}\text{\rm \ \ for all } 0<\alpha<p.
        \]
		In both cases, we obtain the desired inequality, since $\{B(x_k,r_k)\}_{k\in\N}$ is arbitrary. 
  \end{proof}
We obtain the following, by the same argument as in the proof of the previous lemma:
\begin{lemma}\label{lem:Radius-Content}
Let $1<Q<+\infty$, and suppose that $(X,d,\mu)$ is Ahlfors $Q$-regular.  Let $0<\alpha<Q$ and $+\infty>\lambda>1$.  Then for all $+\infty>r>0$, we have that 
\[
\Ha^\alpha_{+\infty}(B(O,\lambda r)\setminus B(O,r/\lambda))\gtrsim r^\alpha,
\]
with comparison constant depending only on $\lambda$ and the Ahlfors $Q$-regularity constant.
\end{lemma}

  For $j\in\N$ and $+\infty>C\ge 1$, let us denote $A_{C 2^{j+1}}:=B(O,C2^{j+2})\setminus B(O,C2^{j+1})$. 
	 We now introduce a notion of thin sets in metric measure spaces. 
	 \begin{definition}\label{def-thin-0905} Let $1<Q<{+\infty}$, and let $1\le p\le Q$.
 A subset $E\subset X$ is said to be \emph{$p$-thin} (or \emph{dyadic $p$-thin at infinity}) if there is a finite constant $C\geq 1$ such that
 	\begin{equation}\label{eq2.4-2805}
		\lim_{k\to{+\infty}}\sum_{j=k}^{+\infty} \frac{\mathcal H^{Q-(p-\alpha)}_{+\infty}(E\bigcap A_{C 2^{j+1}})}{(C 2^{j+1})^\alpha}=0 \text{\rm \ \ for\ all\ $0<\alpha<p$}.
	 \end{equation}
  The finite constant $C\geq 1$ is called the thin constant.  
\end{definition}
When $p=Q$, we will often drop the dependence on $Q$, referring to $Q$-thin sets simply as \emph{thin} sets.

We give some properties of thin sets as below.

	 \begin{lemma} \label{prop2.5-0605}
	 Let  $1<Q<{+\infty}$, and let $1\le p\le Q$. Suppose that $(X,d,\mu)$ is Ahlfors $Q$-regular, and suppose that $E\subset X$ is a $p$-thin set with a thin constant $\lambda\geq 1$. Then the following statements hold:
	 	\begin{enumerate}[label=\arabic*.]
			\item 
   There is an index $j_0\in\mathbb N$ such that  $A_{\lambda 2^{j+1}}\setminus E\ne\varnothing$ for all $j\geq j_0$. 
            \item
We have that
				\begin{equation}\label{eq2.4-0306}
		\lim_{k\to{+\infty}}\sum_{j=k}^{+\infty} \frac{\mu(E\bigcap A_{\lambda 2^{j+1}})}{(\lambda 2^{j+1})^p}=0.
				\end{equation}
			\item  
	{
	  If $X$ has a  weak polar coordinate system $(\Gamma^{\mathcal O}(\mathbb S), \sigma, d_{\mathbb S})$ such that it has a subfamily $\Gamma^{\mathcal O}(\mathbb S')$, where $\mathbb S'\subset\mathbb S$, of pairwise disjoint infinite curves with $\sigma(\mathbb S')>0$  admitting  a dyadic Lipschitz projection $p$ defined as in Section \ref{polar}, then
				\[
					\lim_{k\to{+\infty}}\sum_{j=k}^{+\infty} \mathcal H^{Q-(p-\alpha)}_{+\infty, d_{\mathbb S}} (p(\Gamma^{\mathcal O}(\mathbb S')\cap E\cap A_{\lambda 2^{j+1}}))=0 {\rm \ \ for \ all\ } 0<\alpha<p.
				\]
	}
			\item Let $E$ and $F$ be two $p$-thin sets. Then $E\cup F$ is $p$-thin.
 
		\end{enumerate}
	 \end{lemma}
	 \begin{proof} 
	 \begin{enumerate}[label=\arabic*.]
		\item Suppose that for any $j_0\in\mathbb N$, there is $j\geq j_0$ such that $A_{\lambda 2^{j+1}}\setminus E=\varnothing$.  By Ahlfors $Q$-regularity, we have that 
			\[
			 \frac{\mu(E\bigcap A_{\lambda 2^{j+1}})}{(\lambda 2^{j+1})^p} = \frac{\mu(A_{\lambda 2^{j+1}})}{(\lambda 2^{j+1})^p}\simeq (\lambda 2^{j+1})^{Q-p}\ge1
			\] 
		 which is a contradiction to \eqref{eq2.4-0306}, proved below.
   
        \item The claim follows from Lemma~\ref{lem:Measure-Content Bound} and Definition~\ref{def-thin-0905}. 

		\item 
		{
		Since $p$ is $\frac{C}{\lambda 2^{j+1}}$-Lipschitz on $A_{\lambda 2^{j+1}}\cap \Gamma^{\mathcal O}(\mathbb S')$, we have 
  \[
  \mathcal H^{Q-(p-\alpha)}_{+\infty, d_{\mathbb S}} (p(\Gamma^{\mathcal O}(\mathbb S')\cap E\cap A_{\lambda 2^{j+1}}))\leq \left(\frac{C}{\lambda 2^{j+1}}\right)^{Q-(p-\alpha)} \mathcal H^{Q-(p-\alpha)}_{+\infty}(\Gamma^{\mathcal O}(\mathbb S')\cap E\cap A_{\lambda 2^{j+1}}),
  \]
  see for instance the proof of Theorem 28.4 in \cite{Y15}, and hence the claim follows by  \eqref{eq2.4-2805}.
  }
		\item The claim follows from the subadditivity property of  $\mathcal H^{Q-(p-\alpha)}_{+\infty}$.\qedhere
	\end{enumerate}
	 \end{proof}

\begin{definition}\label{def:ThickSets}
Given a point $O\in X$, an unbounded subset $F\subset X$ is said to be \emph{thick} (or \emph{dyadic $Q$-thick at infinity}) if for every $0<\alpha<Q$,  there are finite  constants $\delta>0$,  $k \in\mathbb N$, and $C\geq 1$ such that
 \begin{equation}\label{thick}
 	\mathcal H^{\alpha}_{+\infty}(F\bigcap A_{C2^{j+1}}) \geq \delta (2^{j+1})^{\alpha}  \text{\rm \ \    for all $j\geq k$. 
	}
 \end{equation}
Here we recall $A_{C2^{j+1}}:=B(O,C2^{j+2})\setminus B(O,C2^{j+1})$.
\end{definition}

\begin{definition}\label{def:almost-thick}
 We say that $F$ is  {\it almost thick} if for every $0<\alpha<Q$, there are a finite sequence $\delta_j>0,$ a number $k\in\mathbb N$, and a constant $C\geq 1$ 
 such that 
 \begin{equation}
 \label{almost-thick}
 \mathcal H^{\alpha}_{+\infty} (F\cap A_{C2^{j+1}}) \geq \delta_j (2^{j+1})^\alpha \text{\rm \ \ for all $j\geq k$.
 }
 \end{equation}
\end{definition}
Clearly, if $F$ is thick then $F$ is almost thick. 

\subsection{Lebesgue points}\label{lebesgue}\ 

Let $f$ be a locally integrable function on a metric measure space $(X,d,\mu)$. Given $0\leq \beta<{+\infty}$ and $0<R\leq {+\infty}$, the {\it $(\beta,R)$-fractional maximal function} of $f$ at $x\in X$, denoted  $\mathcal M_{\beta,R}f(x)$, is defined by 
\[
\mathcal M_{\beta,R} f(x):=\sup_{0<r<R}r^{\beta} \dashint_{B(x,r)}|f|d\mu.
\]

Using standard covering lemmas, the following theorem is obtained in \cite{HKi98}.
\begin{theorem}[{\cite[Lemma 2.6]{HKi98}}]\label{fractional-maximal-theorem}
Suppose that $(X,d,\mu)$ is a metric measure space with metric $d$ and Ahlfors $Q$-regular measure $\mu$ where $1<Q<+\infty$.
 Let $A\subset X$ be a bounded set with $\mu(A)>0$ and the diameter ${\rm diam}(A)$ of $A$ is strictly positive. Then for $\overline\lambda >0,$
	\[
	\mathcal H^{Q-\beta}_{+\infty}\left(\left\{x\in A:  \mathcal M_{\beta,\text{\rm diam}(A)}f(x)>\overline\lambda \right\}\right)\leq 5^{Q-\beta} \frac{2^Q\text{\rm diam}^Q(A)}{\mu(A)}\frac{1}{\overline\lambda} \int_X|f|d\mu \text{\rm \ \ for all $0<\beta<Q$}
	\] and for every  integrable function $f$ on $(X,d,\mu)$.
\end{theorem}	 

For a locally integrable function $u$, we say that $x\in X$ is a \emph{Lebesgue point} of $u$ if $u(x)\in\mathbb R$ and 
\[
u(x)=\lim_{r\to 0}\dashint_{B(x,r)}u\,d\mu.
\]
We define $N_u$ to be the set of non-Lebesgue points in $X$.  The following Lebesgue differentiation theorem holds, see for instance \cite[Page 77]{pekka}:
 \begin{theorem}\label{thm2.3-2809}
 Let $\mu$ be doubling. Then $\mu(N_u)=0$ for all locally integrable functions $u$ on $X$.
 \end{theorem}

 For Sobolev functions, however, this result can be improved:
 

 \begin{proposition}\label{thm2.4-2809}  
 Let $1<Q<+\infty$ and let $1\le p<+\infty$.  Suppose that $(X,d,\mu)$ is a complete, unbounded metric measure space, with  an Ahlfors $Q$-regular measure $\mu$ supporting a $p$-Poincar\'e inequality. If $p>Q$, then $N_u$ is empty for all $u\in\dot N^{1,p}(X)$.  If $1\le p\le Q$, then for every function $u\in\dot N^{1,p}(X)$, we have that $\mathcal H^{Q-(p-\alpha)}(N_u)=0$ for all $0<\alpha<p$.
 \end{proposition}
 
\begin{proof}
If $p>Q$, then by Theorem~\ref{thm:Morrey}, we have that every $u\in\dot N^{1,p}(X)$ is locally H\"older continuous, and so every $x\in X$ is a Lebesgue point of $u$.  

Suppose $1\le p\le Q$.  We first consider the case when $p=1$. Let $u\in \dot N^{1,1}(X)$.  For each $n\in\N$, consider a Lipschitz cutoff function $\eta_n$ supported on $B(O,2n)$ such that $\eta_n\equiv 1$ on $B(O,n)$.  Then $\eta_n u\in N^{1,1}(X)$, and so by \cite[Theorem~4.1, Remark~4.7]{KKST08}, we have that $C_1(N_u\cap B(0,n))=0$.  By subadditivity of the Sobolev $p$-capacity, it follows that $C_1(N_u)=0$. By \cite[Theorem~4.3, Theorem~5.1]{HK10}, it follows that $\Ha^{-1}(N_u)=0$, where $\Ha^{-1}$ is the codimension $1$ Hausdorff  measure, and so by Ahlfors $Q$-regularity we have that $\Ha^{Q-1}(N_u)=0$.

Now consider the case that $1<p\le Q$.  Since $X$ necessarily supports a $q$-Poincar\'e inequality for some $1\le q<p$ by \cite{KZ08}, it follows from \cite[Theorem~5.6.2]{BB15} that $C_p(N_u)=0$ for each $u\in\dot N^{1,p}(X)$.  Here we have also used the same argument as above involving the Lipschitz cutoff functions $\eta_n$.  By \cite[Proposition~3.11]{GKS23}, we then have that $\Ha^{-(p-\alpha)}(N_u)=0$ for all $0<\alpha<p$, where $\Ha^{-(p-\alpha)}$ is the codimension $p-\alpha$ Hausdorff  measure. By Ahlfors $Q$-regularity, we obtain $\Ha^{Q-(p-\alpha)}(N_u)=0$ for all $0<\alpha<p$.
 \qedhere

\end{proof}

Let us recall that  $+\infty>\lambda\geq 1$ is the scaling factor of the $p$-Poincar\'e inequality and we denote  $A_{\lambda 2^{j+1}}:= B(O,\lambda 2^{j+2})\setminus B(O,\lambda 2^{j+1})$.  Given $+\infty>C\ge 1$, we denote $C\cdot A_{\lambda 2^{j+1}}:=B(O, C\lambda 2^{j+3})\setminus B(O, 2^{j}/(\lambda C))$.  We conclude this section by giving two useful lemmas  below. We refer interested readers to \cite{HK98,KN22} for discussions on various versions of these lemmas.

\begin{lemma}\label{lem3.1}
Let  $1<Q<+\infty$ and let $1\le p\le Q$. 
Suppose that $(X,d,\mu)$ is an unbounded metric measure space with metric $d$ and Ahlfors $Q$-regular measure $\mu$ supporting a $p$-Poincar\'e inequality.
	Let  $0<a<+\infty$ and $j\in\mathbb N$. We assume  that $u\in  \dot N^{1,p}(X)$ with an upper gradient $\rho_u\in L^p_\mu(X)$ and assume that $X$ has the annular chain property. Suppose that
	 $E,F$ are two subsets of $A_{\lambda 2^{j+1}}$ such that $|u(x)-u(y)|\geq a$ for all $x\in E, y\in F$, and that
		$|u(x)-u_{B(x,2^j)}|\leq {a}/{5}$, $|u(y)-u_{B(y,2^{j})}|\leq {a}/{5}$ for some $x\in E$, $y\in F$.
		Then 
			\begin{equation}
			\label{eq3.3-1907} 
			1 \lesssim  a^{-p}\int_{c_2\cdot A_{\lambda 2^{j+1}}}\rho_u^pd\mu,
			\end{equation}
   where $c_2\ge 1$ is the finite constant from Definition~\ref{dyadic}.
\end{lemma}
\begin{proof}Fixed $x\in E,y\in F$ such that $|u(x)-u_{B(x,2^j)}|\leq {a}/{5}$, $|u(y)-u_{B(y,2^{j})}|\leq {a}/{5}$, 
we have that
	$
		a\leq |u(x)-u(y)|\leq {a}/{5} + |u_{B(x,2^j)}-u_{B(y,2^j)}|+{a}/{5}.
		$
	Hence
		$
		a\lesssim |u_{B(x,2^j)}-u_{B(y,2^j)}|.
		$
	Let $\{B_i\}_{i=1}^M$  be as the sequence of  finite-chain balls connecting $B(x,2^j)$ and $B(y,2^j)$  where $M$ independent of $j,\lambda$ as in Definition \ref{dyadic}.
	Since $\mu$ is Ahlfors $Q$-regular and supports a $p$-Poincar\'e inequality on $(X,d)$, it follows that 
		\begin{align*}
		a&\lesssim |u_{B(x,2^j)}- u_{B(y,2^j)}|
		\lesssim \sum_{i=1}^M \dashint_{B_i}|u-u_{B_i}|d\mu \lesssim \sum_{i=1}^M 2^j \left (\dashint_{\lambda B_i}\rho_u^pd\mu\right )^{\frac{1}{p}}.
		\end{align*}
	Hence there is an index $i$  such that 
		$a\lesssim 2^j \left (\dashint_{\lambda B_i}\rho_u^pd\mu\right )^{\frac{1}{p}}.
		$
		 By Ahlfors $Q$-regularity and the fact that  $\lambda B_i\subset c_2\cdot A_{\lambda 2^{j+1}}$ for all $1\le i\le M$, it follows that 
   \[
   1\lesssim a^{-p}2^{jp}\dashint_{\lambda B_i}\rho_u^pd\mu\lesssim a^{-p}\int_{c_2\cdot A_{\lambda 2^{j+1}}}\rho_u^pd\mu,
   \]
   which completes the proof.
\end{proof}

\begin{lemma}\label{lem3.2}
Let  $1<Q<+\infty$ and let $1\le p\le Q$. 
Suppose that $(X,d,\mu)$ is a complete, unbounded metric measure space with metric $d$ and Ahlfors $Q$-regular measure $\mu$ supporting a $p$-Poincar\'e inequality.
Let  $0<a<+\infty$ and $j\in \mathbb N$, and let $+\infty>C\ge 1$. We assume that $u\in\dot N^{1,p}(X)$ with an upper gradient $\rho_u\in L^p_\mu(X)$. Suppose that
	 $E$ is a subset of $A_{\lambda 2^{j+1}}$ such that 
	$|u(x)-u_{B(x,2^j)}|> {a}/{5}
	$ for all $x\in E$. Then for every $0<\alpha<p$,
	 \begin{equation}\label{eq3.6}
	 \mathcal{H}^{Q-(p-\alpha)}_{{+\infty}}(E)\lesssim  \frac{1}{(1-2^{-\frac{\alpha}{p}})^p} (2^{j})^\alpha a^{-p}\int_{C\cdot A_{\lambda 2^{j+1}}}\rho_u^{p}d\mu.
	 \end{equation}
\end{lemma}
\begin{proof}Let $x\in E$. We denote $B_i:=B(x, 2^{j-i})$ where $i\in\mathbb N$. 
By Proposition \ref{thm2.4-2809}, we may assume that  every $x\in E$ is a Lebesgue point of $u$. Since  $\mu$ is Ahlfors $Q$-regular and supports a $p$-Poincar\'e inequality, we have that for all $x\in E$,
		\begin{align}
		 \frac{a}{5}< &|u(x)-u_{B(x,2^j)}|\leq \sum_{i=0}^{+\infty} |u_{B_i}-u_{B_{i+1}}|
		\lesssim  \sum_{i=0}^{+\infty} \dashint_{B_{i}}|u-u_{B_i}|d\mu \lesssim \sum_{i=0}^{+\infty} 2^{j-i}\left (\dashint_{\lambda B_i}\rho_u^pd\mu \right )^{1/p}.
		\notag
		\end{align}
	Then
		\[
			a\lesssim \sum_{i=0}^{+\infty} 2^{(j-i)\alpha/p}\left(2^{(j-i)(p-\alpha)}\dashint_{\lambda B_i}\rho_u^pd\mu\right)^{1/p} \lesssim  \frac{1}{1-2^{-\frac{\alpha}{p}}} 2^{j\alpha/p} \mathcal M^{1/p}_{p-\alpha,\lambda 2^j}\rho_u^p(x) \text{\rm \ \ for all $x\in E$.
			}
		\] 
		Hence, we have that 
		  $(1-2^{-\frac{\alpha}{p}})^pa^p2^{-j\alpha}\lesssim \mathcal M_{p-\alpha,\lambda 2^j}\rho_u^p(x)$ for all $x\in E$. By Theorem \ref{fractional-maximal-theorem} applied to  the zero extension $\rho_u$ to $C\cdot A_{\lambda 2^{j+1}}$, 		we obtain that for all $0<\alpha<p$,
		\begin{align*}
		\mathcal H^{Q-(p-\alpha)}_{+\infty}(E)\leq & \mathcal H_{{+\infty}}^{Q-(p-\alpha)}(\left\{ x\in A_{\lambda 2^{j+1}}:  {(1-2^{-\frac{\alpha}{p}})^p}a^p2^{-j\alpha}  \lesssim  \mathcal M_{p-\alpha,\lambda 2^j}\rho_u^p(x)\right\}) \\ 
		\lesssim &  \frac{1}{(1-2^{-\frac{\alpha}{p}})^p}a^{-p} 2^{j\alpha} \int_{C\cdot A_{\lambda 2^{j+1}}}\rho_u^pd\mu
		\end{align*}
		which is \eqref{eq3.6}.
	  This completes the proof.
\end{proof}

\section{Proofs of Theorems~\ref{thm:p<Q Thin}, \ref{thm1.3-2208} and \ref{thm4-2705}}

Throughout this section, we always assume that $+\infty >\lambda\geq 1$ is
the scaling factor of the $p$-Poincar\'e inequality.  For $j\in\N$ and $+\infty>c\ge 1$, we set $A_{\lambda 2^{j+1}}:=B(O,\lambda2^{j+2})\setminus B(O,\lambda2^{j+1})$, and $c\cdot A_{\lambda 2^{j+1}}:=B(O,c\lambda 2^{j+3})\setminus B(O,2^{j}/(c\lambda))$. 
By telescoping-arguments, we obtain Proof of Theorem \ref{thm:p<Q Thin}, Lemma \ref{thm1.2-0505} and Lemma \ref{lem5.1-2308} as below. Our main results in this section are  Lemma \ref{lem5.2-2308} and Lemma \ref{lem5.3-0709}.

Let us begin with the following  analytic lemma, to be used later in this section:

\begin{lemma}
	\label{in-fact}Let $+\infty>\lambda\geq1$, $1\le p<{+\infty}$, and let $u:X\to\R$.  Suppose that there exist $F\subset X$ and $c_0\in[-\infty,+\infty]$ so that $\lim_{x\to+\infty,x\in F}u(x)=c_0$.  If $\{a_k\}_{k\in\N}$ is a sequence with $a_{k}\geq 0$ and such that $\sum_{k\in\mathbb N}a_k<{+\infty}$, then there exist $n_1\in\mathbb N$ and a non-increasing sequence $\{b_k\}_{k\in\mathbb N}$, with $b_k>0$ and $\lim_{k\to{+\infty}}b_k=0$, such that 
	\begin{equation}\label{eq3.1-2805}
	\sum_{k\in\mathbb N}a_kb_k^{-p}<{+\infty}
	\end{equation}
and such that for $k\geq n_1$,
	\[ \begin{cases}
	|u(x)-c_0|<b_k/2 \text{\rm \ \  for all $x\in  F\cap A_{\lambda 2^{k+1}}$} & \text{\rm \ \ if \ } c_0\in \mathbb R,\\
	\text{\rm $u(x)-\frac{1}{b_k}>b_{k}/2$ for all $x\in F\cap A_{\lambda 2^{k+1}}$} & \text{\rm \ \ if \ }c_0=+\infty,\\
	\text{\rm $u(x)+\frac{1}{b_k}<b_{k}/2$ for all $x\in F\cap A_{\lambda 2^{k+1}}$} & \text{\rm \ \ if \ }c_0=-\infty.
	\end{cases}
	\]
\end{lemma}


\begin{proof}
For each $i\in\N$, we have that $\sum_{k\in\N}ia_k<\infty$, and so there exists $n_i\in\N$ such that $\sum_{k=n_i}^{+\infty} ia_k<1/2^i$.  Furthermore, since $\lim_{x\to+\infty,x\in F}u(x)=c_0$, the $n_i$ can be chosen sufficiently large so that for all $k\ge n_i$, we have 
\[
		\begin{cases}
		\text{\rm  $|u(x)-c_0|<1/(2i^{1/p})$ for all $x\in  F\cap A_{\lambda 2^{k+1}}$} & \text{\rm \ \ if $c_0\in\mathbb R$},\\
		\text{\rm $u(x)-i^{1/p}>1/(2i^{1/p})$ for all $x\in  F\cap A_{\lambda 2^{k+1}}$} & \text{\rm \ \ if $c_0=+\infty$},\\ 
		\text{\rm $u(x)+i^{1/p}<1/(2i^{1/p})$ for all $x\in  F\cap A_{\lambda 2^{k+1}}$} & \text{\rm \ \ if $c_0=-\infty$}.
		\end{cases}
	\]
The $n_i$ can also be chosen to be a strictly increasing sequence.  For each $k\in\N$ such that $n_i\le k<n_{i+1}$, we set $b_k=1/i^{1/p}$, and so we can replace $1/(2i^{1/p})$ with $b_k/2$ and $i^{1/p}$ with $1/b_k$ in the previous display.  We have that the $b_k$ are strictly positive, non-increasing, and $\lim_{k\to\infty}b_k=0$. If $k<n_1$, we define $b_k=1$.  Finally,
\[
\sum_{k\in\N}a_kb_k^{-p}
\leq \sum_{k=1}^{n_1}a_k+\sum_{i\in\N}\sum_{k=n_i}^{n_{i+1}}ia_k<\sum_{k\in\mathbb N}a_k+\sum_{i\in\N}\frac{1}{2^i}<+\infty.\qedhere
\]
\end{proof}


We now prove Theorem~\ref{thm:p<Q Thin}.

\begin{proof}[Proof of Theorem~\ref{thm:p<Q Thin}]
    Let $u\in\dot N^{1,p}(X)$, and let $\rho_u$ be a $p$-integrable upper gradient of $u$. Since $(X,d,\mu)$ is complete, Ahlfors $Q$-regular, and satisfies a $p$-Poincar\'e inequality with $p<Q$, it follows from \cite[Theorem~3.3]{K07} that the annular chain property holds.  Therefore, we have from \cite{KN22}, that there exists a constant $c\in\R$ such that 
    \[    \lim_{t\to+\infty}u(\gamma(t))=c\quad\text{for $p$-a.e.\ }\gamma\in\Gamma^{+\infty}.
    \]
    Let $\Gamma_1$ denote the family of curves in $\Gamma^{+\infty}$ for which this holds.  As $p< Q$, it follows that $\Mod_p(\Gamma^{+\infty})>0$, see \cite[Theorem~1.2]{KN22} for example, and so there exists a finite constant $C_1>0$ such that 
    \begin{equation}\label{eq:25-8-23(6)}
    \Mod_p(\Gamma_1)\ge C_1>0.
    \end{equation}

    Fix $k\in\N$.  For each $j\in\N$, define
    \[
    \Gamma_{j,k}:=\{\gamma\in\Gamma_1:|u(x)-c|<2^{-k}\text{ for all }x\in\gamma\setminus B(O,\lambda2^j)\}.
    \]
    We then have that $\Gamma_1=\bigcup_{j\in\N}\Gamma_{j,k}$, and so by subadditivity of the $p$-modulus, 
    \[
    \sum_{j\in\mathbb N}2^{-j}C_1=C_1\leq {\rm Mod}_p(\Gamma_1)\leq\sum_{j\in\mathbb N}{\rm Mod}_p(\Gamma_{j,k}).
    \]
   Hence there exists $j_k\in\N$ such that 
    \begin{align}\label{eq:1-12-24,1}
        \Mod_p(\Gamma_{j_k,k})\ge 2^{-j_k}C_1>0.
    \end{align}
    Then for all $\gamma\in\Gamma_{j_k,k}$,
    \[
    |u(x)-c|<2^{-k} {\rm \ \ for\ all \ }x\in \gamma\setminus B(O,\lambda 2^j) \ {\rm and \ for \ }j\geq j_k.
    \]
Since $\rho_u\in L^p_\mu(X)$,    we also choose $i_k\ge j_k$ sufficiently large so that
    \begin{equation}\label{eq:(1-25)i_k decay}
        2^{j_k}2^{kp}\int_{X\setminus B(O,2^{i_k}/c_2)}\rho_u^p\,d\mu<2^{-k}.
    \end{equation}
    With sequences $\{i_k\}_{k\in\mathbb N}$ and $\{j_k\}_{k\in\mathbb N}$ now fixed, we define a sequence $\{a_l\}_{l\in\N}$ by setting $a_l=2^{-k}$ for $i_k\le l<i_{k+1}$.
    Then for each $k\in\N$, we have that   
    
    \begin{equation}\label{eq3.4-1401}
        |u(x)-c|<a_l \text{\rm \ \ for all $x\in \gamma\cap A_{\lambda 2^{l+1}}$, $i_k\leq l< i_{k+1}, \gamma\in \Gamma_{j_k,k}$}.
    \end{equation}    

     For each $k\in\N$ and $i_k\leq l<i_{k+1}$, we define 
    \[
    \text{$F_l:=\bigcup_{\gamma\in\Gamma_{j_k,k}}(\gamma\cap A_{\lambda 2^{l+1}})$ \ and\  }
    E_{l}:=\{x\in A_{\lambda 2^{l+1}}:|u(x)-c|>2a_l\}.
    \]
    We consider three cases.  Firstly, suppose that there exist $x\in E_{l}$ and $y\in F_l\cap A_{\lambda 2^{l+1}}$ such that $|u(x)-u_{B(x,2^l)}|\le a_l/5$ and $|u(y)-u_{B(y,2^l)}|\le a_l/5$.  Then, by Lemma~\ref{lem3.1}, it follows that
    \begin{equation}\label{eq:1-12-24,2}
        1\lesssim a_l^{-p}\int_{c_2\cdot A_{\lambda 2^{l+1}}}\rho_u^p\,d\mu=2^{kp}\int_{c_2\cdot A_{\lambda 2^{l+1}}}\rho_u^p\,d\mu.
    \end{equation}
    For the second case, suppose that $|u(y)-u_{B(y,2^l)}|\ge a_l/5$ for all $y\in F_l\cap A_{\lambda 2^{l+1}}$. Let $0<\alpha<p$. Then by Lemma~\ref{lem:Measure-Content Bound} and Lemma~\ref{lem3.2}, it follows that 
    \begin{equation*}
        \frac{\mu(F_l)}{(2^{l+1})^p}\lesssim\frac{\Ha^{Q-(p-\alpha)}_{+\infty}(F_l)}{(\lambda 2^{l+1})^\alpha}\lesssim \frac{1}{(1-2^{-\frac{\alpha}{p}})^p} 2^{kp}\int_{c_2\cdot A_{\lambda 2^{l+1}}}\rho_u^pd\mu.
    \end{equation*}
    Since we are considering $i_{k+1}\ge l\ge i_k\geq j_k$, we have that $\Mod_p(\Gamma_{j_k,k})\leq \mu(F_l)/(\lambda 2^{l+1})^p$ because for all measurable subsets $A\supset F_l$, $\chi_{A}/(\lambda 2^{l+1})$ is admissible for computing $\Mod_p(\Gamma_{j_k,k})$.  Hence, from \eqref{eq:1-12-24,1}, we obtain
    \begin{equation}\label{eq:1-12-24,3}
       C_1\lesssim \frac{1}{(1-2^{-\frac{\alpha}{p}})^p} 2^{j_k}2^{kp}\int_{c_2\cdot A_{\lambda 2^{l+1}}}\rho_u^pd\mu.
    \end{equation}
    For the third case, suppose that $|u(x)-u_{B(x,2^l)}|\ge a_l/5$ for all $x\in E_{l}$.  Then by Lemma~\ref{lem3.2}, we have that
    \begin{equation}\label{eq:1-12-24,4}
        \frac{\Ha^{Q-(p-\alpha)}_{+\infty}(E_{l})}{(\lambda 2^{l+1})^\alpha}\lesssim \frac{1}{(1-2^{-\frac{\alpha}{p}})^p} 2^{kp}\int_{c_2\cdot A_{\lambda 2^{l+1}}}\rho_u^pd\mu.
    \end{equation}
    
    By \eqref{eq:(1-25)i_k decay}, we have that 
    \begin{equation}\label{eq:(1-25)i_k decay II}
    \sum_{k\in\N}2^{j_k}2^{kp}\int_{X\setminus B(O,2^{i_k}/c_2)}\rho_u^p\,d\mu<+\infty.
    \end{equation}
    Therefore, there exists $k_0\in\mathbb N$, sufficiently large, so that the two estimates \eqref{eq:1-12-24,2} and \eqref{eq:1-12-24,3} fail for all $k\geq k_0$ and $i_k\le l<i_{k+1}$.  Thus, for all $k\ge k_0$ and $i_k\le l<i_{k+1}$, it follows that \eqref{eq:1-12-24,4} holds.  
    
    Let $E:=\bigcup_{l=i_{k_0}}^{+\infty}E_{l}$.  For $m\ge i_{k_0+1}$, let $k_m\in\N$ be the smallest positive integer such that $i_{k_m}\le m$.  By \eqref{eq:1-12-24,4}, we then have 
    that 
    \begin{align*}
        \sum_{n=m}^{+\infty}\frac{\Ha^{Q-(p-\alpha)}_{+\infty}(E\cap A_{\lambda 2^{n+1}})}{(\lambda 2^{n+1})^\alpha}&\le\sum_{n=i_{k_m}}^{+\infty}\frac{\Ha^{Q-(p-\alpha)}_{+\infty}(E\cap A_{\lambda 2^{n+1}})}{(\lambda 2^{n+1})^\alpha}=\sum_{k=k_m}^{+\infty}\sum_{l=i_{k}}^{i_{k+1}-1}\frac{\Ha^{Q-(p-\alpha)}_{+\infty}(E_l)}{(\lambda 2^{l+1})^\alpha}\\
        &\lesssim\frac{1}{(1-2^{-\frac{\alpha}{p}})^p} \sum_{k=k_m}^{+\infty}\sum_{l=i_{k}}^{i_{k+1}-1}2^{kp}\int_{c_2\cdot A_{\lambda 2^{l+1}}}\rho_u^pd\mu\\
        &\lesssim \frac{1}{(1-2^{-\frac{\alpha}{p}})^p} \sum_{k=k_m}^{+\infty}2^{kp}\int_{X\setminus B(O,2^{i_k}/c_2)}\rho_u^p\,d\mu.
    \end{align*}
    Hence, from \eqref{eq:(1-25)i_k decay II}, we have that 
    \[
    \lim_{m\to+\infty}\sum_{n=m}^{+\infty}\frac{\Ha^{Q-(p-\alpha)}_{+\infty}(E\cap A_{\lambda 2^{n+1}})}{(\lambda 2^{n+1})^\alpha}=0 {\rm \ \ for\ all\ 0<\alpha<p},
    \]
    and so $E$ is a $p$-thin set.  By the definition of $E$,  it follows that 
    \[
    \lim_{E\not\ni x\to+\infty}u(x)=c.
    \]

 Now suppose that there exist  $c'\in\R$ and a $p$-thin set $E'\subset X$ such $\lim_{E'\not\ni x\to+\infty}u(x)=c'$.  Let $F:=X\setminus E$ and $F':=X\setminus E'$.  We claim that $A_{\lambda 2^{j+1}}\cap F\cap F'\ne\varnothing$ for all sufficiently large $j\in\N$.  Indeed, by Lemma~\ref{prop2.5-0605} (4.), it follows that $E\cup E'$ is a $p$-thin set, and so by Lemma~\ref{prop2.5-0605}, it follows that for all sufficiently large $j$, 
 \[
 A_{\lambda 2^{j+1}}\cap F\cap F'=A_{\lambda 2^{j+1}}\setminus (E\cup E')\ne\varnothing.
 \]
 Therefore, we must have that $c=c'$.
 \end{proof}

We now turn our attention to the proofs of  Theorem~\ref{thm1.3-2208} and Theorem~\ref{thm4-2705}, where we consider functions $u\in\dot N^{1,p}(X)$ with $p=Q$.  With the following lemma, we first show that if $u$ has a limit along some infinite curve, then it attains the same limit outside of some thin set.  However, since $p=Q$, the family of infinite curves has zero $p$-modulus, and so we must modify the argument used in the proof of Theorem~\ref{thm:p<Q Thin}.  Recall that when $p=Q$, we often drop the dependence on $Q$, referring to $Q$-thin sets simply as thin sets.

 \begin{lemma}\label{thm1.2-0505} 
Let  $1<Q<+\infty$. 
Suppose that $(X,d,\mu)$ is a complete, unbounded metric measure space with metric $d$ and Ahlfors $Q$-regular measure $\mu$ supporting a $Q$-Poincar\'e inequality. 
Let $u\in\dot N^{1,Q}(X)$ and suppose that there is an infinite curve $\gamma\in\Gamma^{+\infty}$ so that  $ \lim_{t\to+\infty}u(\gamma(t))=c_0\in[-\infty,+\infty]$. Then, there exists a thin set $E$ such that 
			\begin{equation}\label{eq1.2-0905}
				\lim_{E\not\ni x\to{+\infty}}u(x) =c_0.
			\end{equation}
\end{lemma}

\begin{proof}


We recall that  $+\infty>\lambda\geq 1$ is
the scaling factor of the $Q$-Poincar\'e inequality, and that  $\rho_u$ is  a $Q$-integrable upper gradient of $u$. We also recall that we use the notation $A_{\lambda 2^{j+1}}:=B(O,\lambda2^{j+2})\setminus B(O,\lambda2^{j+1})$ and $c_2\cdot A_{\lambda 2^{j+1}}:=B(O,c_2\lambda 2^{j+3})\setminus B(O,2^{j}/(c_2\lambda))$ for $j\in\mathbb N$, where   $c_2$ is as in Definition \ref{dyadic}.  Here the annular chain property  is given since $X$ is complete, see Section \ref{sec-dyadic}.

We first prove \eqref{eq1.2-0905} for the case of $c_0\in\mathbb R$.
Since $u\in\dot N^{1,Q}(X)$, there exists by Lemma \ref{in-fact}, a non-increasing sequence $\{a_j\}_{j\in\mathbb N}$ with $a_j>0$ such that 
\begin{equation}
\label{sequence-a_j}\lim_{j\to{+\infty}}a_j=0,\ \sum_{j=1}^{+\infty} a_j^{-Q}\int_{c_2\cdot A_{\lambda 2^{j+1}}}\rho_u^Qd\mu<{+\infty}, \text{\rm \ \ and\ \ } |u-c_0|<a_j/2 \text{\rm \ \ on \ }{ \gamma\cap A_{\lambda 2^{j+1}}}\ {\rm for\ }j\in\mathbb N.
\end{equation}
Let $E_l:=\{ x\in A_{\lambda 2^{l+1}}: |u(x)-c_0|>a_l\}$ for $l\in\mathbb N$.  We set $E:=\bigcup_{l=1}^{+\infty} E_l$. It is clear that $\lim_{x\to{+\infty}, x\notin E}u(x)=c_0$. We now only need to prove that $E$ is a thin set. Since $a_l/2\leq |u(x)-c_0| - |u(y)-c_0|\leq |u(x)-u(y)|$ for all $x\in E_l$ and $y\in  \gamma\cap A_{\lambda 2^{l+1}}$ by \eqref{sequence-a_j}, we have that
	\begin{equation}\label{eq3.9-1207}
 a_l/2\leq |u(x)-u(y)| \leq | u(x)-u_{B(x,2^l)}|+|u_{B(x,2^l)}-u_{B(y,2^l)}|+|u_{B(y,2^l)}-u(y)|
 	\end{equation}
  for all $x\in E_l$ and $y\in  \gamma\cap A_{\lambda 2^{l+1}}$.
Using Lemma \ref{lem3.1} and Lemma \ref{lem3.2}, we consider the following cases:

If  there exist $x\in E_l$ and $y\in\gamma\cap A_{\lambda 2^{l+1}}$ such that $|u(x)-u_{B(x,2^l)}|\le\frac{a_l}{10}$ and $|u(y)-u_{B(y,2^l)}|\le\frac{a_l}{10}$, then by estimate \eqref{eq3.3-1907} from Lemma \ref{lem3.1}, we have that 
\begin{equation}
	\label{eq3.11} 1\lesssim a_l^{-Q}\int_{c_2\cdot A_{\lambda 2^{l+1}}}\rho_u^Qd\mu. 	\end{equation}
Let $0<\alpha<Q$.
If $|u(x)-u_{B(x,2^l)}|>\frac{a_l}{10}$ for all $x\in E_l$, then estimate \eqref{eq3.6} from Lemma \ref{lem3.2} gives
 \begin{equation}
		\label{eq3.12} 
		\frac{\mathcal H^\alpha_{+\infty}(E_l)}{(2^l)^\alpha}\lesssim  \frac{1}{(1-2^{-\frac{\alpha}{Q}})^Q} a_l^{-Q}\int_{c_2\cdot A_{\lambda 2^{l+1}}}\rho_u^Qd\mu.
\end{equation}
Likewise, if $|u(y)-u_{B(y,2^l)}|>\frac{a_l}{10}$ for all $y\in \gamma\cap A_{\lambda2^{l+1}}$, then the same estimate from Lemma~\ref{lem3.2}, with $\alpha=1$, gives  
\begin{equation}\notag
			\frac{ \mathcal{H}^{1}_{+\infty} (\gamma\cap A_{\lambda 2^{l+1}})}{2^l}\lesssim a_l^{-Q}\int_{c_2\cdot A_{\lambda 2^{l+1}}}\rho_u^Qd\mu.
\end{equation} 
  By Lemma \ref{lem2.3-0505}, $\mathcal H^{1}_{+\infty} (\gamma\cap A_{\lambda 2^{l+1}})\gtrsim 2^{l}$ for all sufficiently large $l$, and so in this case, we also have that 	
	 \begin{equation}
	 		\label{eq3.13} 1\lesssim  a_l^{-Q}\int_{c_2\cdot A_{\lambda 2^{l+1}}}\rho_u^Qd\mu
	 \end{equation} 
  for all sufficiently large $l$.
  
  By \eqref{sequence-a_j}, there exists $k\in\N$ such that \eqref{eq3.13} does not  hold for any $l\geq k$, and so for all $l\ge k$, there exists $y\in\gamma\cap A_{\lambda 2^{l+1}}$ such that $|u(y)-u_{B(y,2^l)}|\le\frac{a_l}{10}$.  From \eqref{eq3.9-1207} and \eqref{eq3.11}, we then have that \eqref{eq3.12} holds true for all $l\geq k$. 
	  	Summing \eqref{eq3.12}  over $l\geq k$, it follows from \eqref{sequence-a_j} that
	 			\begin{equation}\label{eq4.6-2905}
	 		\lim_{k\to{+\infty}}\sum_{l=k}^{+\infty}
	 		 \frac{\mathcal H^{\alpha}_{{+\infty}}(E_l)}{(2^{l})^\alpha}\lesssim  \frac{1}{(1-2^{-\frac{\alpha}{Q}})^Q}  \lim_{k\to{+\infty}}\sum_{l=k}^{+\infty} a_l^{-Q}\int_{c_2\cdot A_{\lambda 2^{l+1}}}\rho_u^Qd\mu=0,
	 			\end{equation}
	 		and so $E=\bigcup_{l=1}^{+\infty} E_l$ is a thin set. This completes the proof for the case $c_0\in\mathbb R$.
	
	Now suppose that $c_0=+\infty$.  We use similar arguments as above.  By Lemma \ref{in-fact} and since $u\in \dot N^{1,p}(X)$, there exists a non-increasing sequence $\{a_j\}_{j\in\mathbb N}$ with $a_j>0$ such that 
	\[
		\lim_{j\to+\infty}a_j=0, \ \sum_{j=1}^{+\infty} a_j^{-Q}\int_{c_2\cdot A_{\lambda 2^{j+1}}}\rho_u^Qd\mu<+\infty, \text{\rm \ \ and \ \  $u-\frac{1}{a_j}>a_{j}/2$ on  $\gamma\cap A_{\lambda 2^{j+1}}$ for $j\in\mathbb N$}.
	\]
Letting $E_{l}:=\{x\in A_{\lambda 2^{l+1}}: u(x)-\frac{1}{a_l}< \frac{a_l}{4}\}$ for $l\in\mathbb N$,  we again set $E:=\bigcup_{l=1}^{+\infty}E_l$. It is clear that $\lim_{x\to+\infty, x\notin E}u(x)=+\infty$. Moreover, $0<a_l/4\leq (u(y)-\frac{1}{a_l}) - (u(x)-\frac{1}{a_l})=u(y)-u(x)\leq |u(y)-u(x)|$ for all $x\in E_l$ and $y\in \gamma\bigcap A_{\lambda 2^{l+1}}$. Hence 
	\[
	a_l/4\leq |u(x)-u(y)| \leq | u(x)-u_{B(x,2^l)}|+|u_{B(x,2^l)}-u_{B(y,2^l)}|+|u_{B(y,2^l)}-u(y)|
 	\]
  for all $x\in E_l$ and $y\in  \gamma\cap A_{\lambda 2^{l+1}}$. Repeating the above arguments from \eqref{eq3.9-1207} to \eqref{eq4.6-2905}, we obtain that $E$ is a thin set. 
Similarly, it is easy to check that the  claim \eqref{eq1.2-0905} holds when $c_0=-\infty$. This completes the proof.
\end{proof}
	

\begin{lemma}\label{lem4.5-0703}
Let  $1<Q<+\infty$. 
Suppose that $(X,d,\mu)$ is an  unbounded metric measure space with metric $d$ and Ahlfors $Q$-regular measure $\mu$. Let $E$ be a thin set. Then $F:=X\setminus E$ is a thick set.
\end{lemma}
\begin{proof}
Let $E$ be a $Q$-thin set. Let $0<\alpha<Q$. Then there is a finite constant $C\geq 1$ such that 
\begin{equation}\label{eq4.16-0803}
\lim_{k\to{+\infty}}\sum_{j=k}^{+\infty}\frac{\mathcal H^{\alpha}_{+\infty}(E\cap A_{C2^{j+1}})}{(C 2^{j+1})^\alpha}=0. 
\end{equation}
By Lemma~\ref{lem:Radius-Content}, there  is a constant $C_1>0$ such that  
\begin{equation}\label{eq4.18-0803}
	\mathcal H_{+\infty}^{\alpha}( A_{C2^{j+1}})\geq C_1 (2^{j+1})^\alpha \text{\rm \ \ for all $j\in\mathbb N$},
\end{equation}
and by \eqref{eq4.16-0803}, there is an index $k\in\mathbb N$ sufficiently large so that 
\begin{equation}\label{eq4.19-0803}
	 \frac{\mathcal H^\alpha_{+\infty}(E\cap A_{C2^{j+1}})}{(2^{j+1})^\alpha}<\frac{C_1}{2}
\end{equation}
for all $j\geq k$.
Combining the above estimate \eqref{eq4.19-0803} with \eqref{eq4.18-0803},  we obtain from subadditivity of $\mathcal H^{\alpha}_{+\infty}$ that for $j\geq k$,
\begin{align*}
	\mathcal H^{\alpha}_{+\infty}((X\setminus E) \cap A_{C2^{j+1}}) \geq& \mathcal H^{\alpha}_{+\infty} ( A_{C2^{j+1}}) - \mathcal H^{\alpha}_{+\infty}( E\cap A_{C2^{j+1}})\\
\geq & C_1 (2^{j+1})^\alpha-\frac{C_1}{2}(2^{j+1})^\alpha=\frac{C_1}{2}(2^{j+1})^\alpha.
\end{align*}
 Hence, $X\setminus E$ is a thick set, completing the proof.
\end{proof}

 	In what follows of this section, we let  $A_j:=A_{C2^{j+1}}$ for $j\in\mathbb N$, where $C$ satisfies the condition \eqref{thick} with respect to the definition for a given thick set $F$.

\begin{lemma}\label{lem5.1-2308}
Let  $1<Q<+\infty$. 
Suppose that $(X,d,\mu)$ is an  unbounded metric measure space with metric $d$ and Ahlfors $Q$-regular measure $\mu$ supporting a $Q$-Poincar\'e inequality. Additionally, we assume that $X$ is complete.
Let $F\subset X$ be a thick set. Then there exist finite constants $k_F\in\mathbb N$ and $C_F>0$ so that  
\begin{equation}\label{eq5.1-2308}
\text{\rm Cap}_Q(F\cap A_{j}, A_{j+2}, G_j)\geq C_F \text{\rm \ \ for all $j\geq k_F$}
\end{equation}
where  $G_j:= B(O,\lambda\cdot \overline c_2 \cdot C\cdot 2^{j+4})\setminus B(O,\frac{2^{j-2}}{\lambda \cdot \overline c_2})$. Here $\lambda$ is the scaling factor of the $Q$-Poincar\'e inequality, $\overline c_2$ is the constant from Proposition~\ref{prop2.3-0806}, $C$ satisfies the condition \eqref{thick} with respect to the thick set $F$, and $A_j:=A_{C2^{j+1}}=B(O,C2^{j+2})\setminus B(O,C2^{j+1})$. 
\end{lemma}
\begin{proof}
Let $0<\alpha<Q$.
Since $F$ is a thick set,  \eqref{thick} gives that there are finite constants $\delta_F>0$, and $k_F\in\N$ so that 
 \begin{equation}
 \label{eq4.1-2308}
 	\mathcal H^{\alpha}_{+\infty}(F\cap A_j)\geq \delta_F (2^{j+1})^\alpha
 \end{equation}
 holds for all 
   $j\ge k_F$. 
 Moreover, from Lemma~\ref{lem:Radius-Content} we have that 
 \begin{equation}
 \label{eq4.2-2308}
 	\mathcal H^{\alpha}_{+\infty}(A_{j+2})\gtrsim (2^{j+2})^\alpha.
 \end{equation} 
 
 Let $g$ be a $Q$-integrable upper gradient in $G_j$ of a function $u$ which satisfies  $u|_{F\cap A_j}\equiv 1$ and $u|_{A_{j+2}}\equiv 0$. 
 Then for all $x\in F\cap A_j$ and $y\in A_{j+2}$, we have
 \begin{equation}\label{eq4.2-0905}
 	1\leq |u(x)-u(y)|\leq |u(x)-u_{B(x, 2^j)}|+|u_{B(x,2^j)}-u_{B(y, 2^{j})}|+|u_{B(y, 2^{j})}-u(y)|.
 \end{equation}
Since $X$ is complete by assumption, our space has the  annular chain property, see Section \ref{sec-dyadic}.  Let $\overline c_2\ge 1$ be the constant from Proposition~\ref{prop2.3-0806}, which appears in the definition of $G_j$ above. Then by Lemma \ref{lem3.2}, we obtain 
 	\[
		\mathcal H^{\alpha}_{+\infty}(F\cap A_j)\lesssim  \frac{1}{(1-2^{-\frac{\alpha}{Q}})^Q}  (2^j)^\alpha \int_{G_j}g^Qd\mu
	\]
 if $\frac{1}{5}\leq |u(x)-u_{B(x, 2^j)}|$ for all $x\in F\cap A_j$.  Likewise Lemma \ref{lem3.2} gives us that 
 	\[
	\mathcal H^{\alpha}_{+\infty}(A_{j+2})\lesssim  \frac{1}{(1-2^{-\frac{\alpha}{Q}})^Q}  (2^{j})^\alpha \int_{G_j}g^Qd\mu
	\]
 if $\frac{1}{5}\leq |u_{B(y, 2^{j})}-u(y)|$ for all $y\in A_{j+2}$. Combining these estimates with \eqref{eq4.1-2308} and \eqref{eq4.2-2308}, we obtain \eqref{eq5.1-2308} for all $j\ge k_F$ when either $\frac{1}{5}\leq |u(x)-u_{B(x, 2^j)}|$ for all $x\in F\cap A_j$ or $\frac{1}{5}\leq |u_{B(y, 2^{j})}-u(y)|$ for all $y\in A_{j+2}$. 
 
 Fix $j\ge k_F$ and suppose that there exists $x\in F\cap A_j$ and $y\in A_{j+2}$ such that $|u(x)-u_{B(x,2^j)}|<\frac{1}{5}$ and $|u(y)-u_{B(y,2^j)}|<\frac{1}{5}$.  Then by \eqref{eq4.2-0905}, we have $\frac{3}{5}\le|u_{B(x,2^j)}-u_{B(y, 2^{j})}|.$
By Proposition~\ref{prop2.3-0806}, we can connect the balls $B(x,2^j)$ and $B(y,2^j)$ by a finite chain of balls $\{B_i\}_{i=0}^{\overline M}$, each having radius comparable to $2^j$. Then, by the $Q$-Poincar\'e inequality and Ahlfors $Q$-regularity, it follows that 
 \[
 \frac{3}{5}\leq |u_{B(x,2^j)}-u_{B(y, 2^{j})}|\leq \sum_{i=0}^{{\overline M}-1}|u_{B_i+1}-u_{B_i}|\lesssim \sum_{i=0}^{\overline{M}} \left(\int_{\lambda \cdot B_i}g^Qd\mu \right)^{1/Q}
 \]
 where $\lambda$ is the scaling constant of $Q$-Poincar\'e inequality.  By our choice of $G_j$, it follows that each $\lambda B_i$ is a subset of $G_j$.  By this fact and the above estimate, it follows that there is an index $i$ so that 
 \[
 	1\lesssim \int_{\lambda B_i}g^Qd\mu \leq \int_{G_j}g^Qd\mu.
 \]
Since $(u,g)$ is an arbitrary admissible pair for computing $ \text{\rm Cap}_Q(F\cap A_{j}, A_{j+2}, G_j)$, this estimates gives \eqref{eq5.1-2308} for all $j\ge k_F$ when $|u_{B(x,2^j)}-u_{B(y, 2^{j})}|\geq \frac{3}{5}$ for some $x\in F\cap A_j$ and $y\in A_{j+2}$. This completes the proof.
\end{proof}

Given a rectifiable curve $\gamma$, we denote the length of $\gamma$ by $\ell(\gamma)$.

\begin{lemma}\label{lem5.2-2308}
Let  $1<Q<+\infty$. 
Suppose that $(X,d,\mu)$ is an  unbounded metric measure space with metric $d$ and Ahlfors $Q$-regular measure $\mu$ supporting a $Q$-Poincar\'e inequality. We assume that $X$ is complete.
Let $u\in\dot N^{1,Q}(X)$, and let $F$ be a thick set such that $\lim_{x\to+\infty,x\in F}u(x)=c$ for some $c\in[-\infty,+\infty]$. Then there is a sequence of rectifiable curves, denoted $\{\gamma_j\}_{j\in\mathbb N}$, with $\gamma_j\cap F\neq\emptyset$ and $\gamma_j$ connecting $F\cap A_j$ and $A_{j+2}$ in $G_j$, such that \[
\lim_{\cup_{j\in\mathbb N}\gamma_j\ni x\to+\infty}u(x)=c.
\]
Furthermore, $\ell (\gamma_j)\simeq 2^j$ for each $j\ge k_F$, where $k_F$ is as given in Lemma \ref{lem5.1-2308}.
Here, $G_j$ is defined as in Lemma \ref{lem5.1-2308} as well.
\end{lemma}

\begin{proof}
By Lemma \ref{lem5.1-2308}, there exist  constants $C_F>0$ and $k_F\in\mathbb N$ so that
\begin{equation}\label{eq5.5-0709}
\text{\rm Cap}_Q(F\cap A_{j}, A_{j+2}, G_j)\geq C_F
\end{equation} 
for $j\ge k_F$.
Let $\Gamma_j$ be the family of all rectifiable curves $\gamma$ in $G_j$ connecting $F\cap A_j$ and $A_{j+2}$. 
We may assume that $c=0$, since the proof for $c\neq 0$ is the same.

Since $\lim_{x\to+\infty,x\in F}u=0$, there is a sequence $\varepsilon_j>0$ with $\varepsilon_j\to0$ as $j\to+\infty$ such that  $|u|_{F\cap G_j}|< \varepsilon_j$.  Let $\varepsilon'_j>0$ be defined by 
\begin{equation}\label{eq5.4-2308}
\varepsilon'_j= \varepsilon_j+\left(\frac{2}{C_F} \int_{G_j}g^Qd\mu \right)^{1/Q}
\end{equation}
where $g\in L_\mu^Q(X)$ is an upper gradient of $u$, and let
$\Gamma_j^c:=\{\gamma\in\Gamma_j:|u|_{\gamma}|<\varepsilon_j'\}$.
For each $\gamma\in \Gamma_j\setminus\Gamma_j^c$, there exists $x\in\gamma\cap G_j$ such that $|u(x)|\ge\varepsilon'_j$.  Hence, it follows that 
\[
\varepsilon'_j-\varepsilon_j\leq \sup_{x\in\gamma}u(x)-\inf_{x\in\gamma}u(x)\leq\int_\gamma gds
\]
for all $\gamma\in\Gamma_j\setminus\Gamma_j^c$.  As such, $\frac{g}{\varepsilon'_j-\varepsilon_j}\chi_{G_j}$ is admissible for computing $\text{\rm Mod}_Q(\Gamma_j\setminus\Gamma_j^c)$. 
By \eqref{eq2.4-2308},\eqref{eq5.5-0709}, and \eqref{eq5.4-2308}, it follows that for $j\ge k_F$,
\begin{align}\label{eq5.7-0709}
C_F&\leq \text{\rm Cap}_Q(F\cap A_{j}, A_{j+2},G_j)=\text{\rm Mod}_Q(\Gamma_j)\leq\text{\rm Mod}_Q(\Gamma_j^c)+\text{\rm Mod}_Q(\Gamma_j\setminus\Gamma_j^c)\nonumber\\   
    &\le\text{\rm Mod}_Q(\Gamma_j^c)+\frac{1}{(\varepsilon'_j-\varepsilon_j)^Q}\int_{G_j}g^Qd\mu=\text{\rm Mod}_Q(\Gamma_j^c)+\frac{C_F}{2},
\end{align}
  and so we have that $
    \text{\rm Mod}_Q(\Gamma_j^c)\ge C_F/2.$ Here the $Q$-modulus of families is defined on $G_j$.
 In particular, $\Gamma_j^c$ is nonempty for such $j$.
Let $\gamma_j\in\Gamma_j^c$.  Noting that $\varepsilon'_j\to 0$ as $j\to\infty$, we have that
\[
\lim_{x\to+\infty,\,x\in\cup_{j\in\mathbb{N}}\gamma_j}u(x)=0.
\]

 Now for $L>0$, let $\Gamma^c_{j,L}:=\{\gamma\in\Gamma^c_j:\ell (\gamma)>L\}$.  Then, $\chi_{G_j}/L$ is admissible for computing $\Mod_Q(\Gamma_{j,L}^c)$, and so by the Ahlfors $Q$-regularity, we have 
\[
\Mod_Q(\Gamma^c_{j,L})\le\mu(G_j)/L^Q\le C_{Q}(2^j/L)^Q
\]
where $C_{Q}$ is the constant of Ahlfors $Q$-regularity.  Thus, 
$\Mod_Q(\Gamma^c_{j,L})<C_F/4$
if $L>(4C_{Q}/C_F)^{1/Q}2^j$.  Since $\Mod_Q(\Gamma^c_j)\ge C_F/2$ for $j\ge k_F$, it then follows that for such $j$, 
\begin{equation}\label{eq3.19-0307}
\Mod_Q(\{\gamma\in\Gamma_j^c:\ell(\gamma)\le(4C_{Q}/C_F)^{1/Q}2^j\})\ge C_F/4.
\end{equation}
As such, we can choose $\gamma_j$ so that $\ell(\gamma_j)\simeq \ell(\gamma_j\cap A_j)\simeq 2^j$. This completes the proof.
\end{proof}
In what follows, given a curve family $\Gamma$, we denote by $|\Gamma|$ the union of the trajectories of the curves in $\Gamma$.
\begin{lemma}
\label{lem5.3-0709} 
Let  $1<Q<+\infty$. 
Suppose that $(X,d,\mu)$ is an  unbounded metric measure space with metric $d$ and Ahlfors $Q$-regular measure $\mu$ supporting a $Q$-Poincar\'e inequality. We assume that $X$ is complete.  Let $u\in \dot N^{1,Q}(X)$, and let $F$ be a thick set such that $\lim_{x\to +\infty,x\in F}u(x)=c$ for some $c\in[-\infty,+\infty]$. Then there is a family  of infinite curves, denoted by $\Gamma^c$, such that \[
\lim_{t\to+\infty}u(\gamma(t))=c
\] for each $\gamma\in\Gamma^c$.  Furthermore for sufficiently large $R>0$, 
the family $\Gamma^c$ satisfies the following properties:
\begin{enumerate}
    \item [{\rm 1.}] $ \mu(B(O,R)\cap|\Gamma^c|)\simeq \mu(A_R\cap |\Gamma^c|)\simeq\mu(B(O,R))$.
     \item [{\rm 2.}] $\ell(\gamma\cap B(O,R))\simeq \ell(\gamma\cap A_R)\simeq R$ for each $\gamma\in\Gamma^c$,
\end{enumerate}
where $A_R:=B(O,R)\setminus B(O,R/2)$.
\end{lemma}
 \begin{proof} 
 We may assume that the constant limit $c$ of $u$ along $F$ equals to 0, since the proof for $c\neq 0$ is the same.  

For each $\mathbb{N}\ni j\geq k_F$, let $\gamma_j$ and $\gamma_{j+10}$ be the curves given by Lemma~\ref{lem5.2-2308}, and let $G_j':=\cup_{i={j-1}}^{j+11}G_i$. We now let $\Gamma_j:=(\gamma_j,\gamma_{j+10}, G_j')$ be the collection of all rectifiable curves in $G_j'$ connecting $\gamma_j$ and $\gamma_{j+10}$.  By Theorems~\ref{thm2.6-0709} and \ref{thm2.7-0709},  there exists a constant $C'>0$ such that $\text{\rm Mod}_Q{(\Gamma_j)}\ge C'$ for all $j$.  Moreover, since $\lim_{\cup_{j\in\mathbb N}\gamma_j\ni x\to+\infty}u(x)=0$, there exists $\varepsilon_j>0$ with $\varepsilon_j\to 0$ as $j\to\infty$ such that $|u|_{\gamma_j\cup\gamma_{j+10}}|<\varepsilon_j$.  Define
\begin{equation}\notag\label{eq5.5-2308}
	\varepsilon'_j= \varepsilon_j+ \left(\frac{100}{{\rm Mod}_Q(\Gamma_j)} \int_{G_j'}g^Qd\mu \right),
\end{equation}
where $g\in L^Q_\mu(X)$ is an upper gradient of $u$.  Note that $\varepsilon'_j\to 0$ as $j\to+\infty$.


Letting $\Gamma_{j,c}:=\{\gamma\in \Gamma_j:|u|_\gamma|<\varepsilon'_j\}$, it follows from a similar argument as in Lemma~\ref{lem5.2-2308} that there is a constant $C_0>0$ such that $\Mod_Q(\Gamma_{j,c})\ge C_0/2$.  Again by a similar argument as in Lemma~\ref{lem5.2-2308}, there exists a constant $C_1>0$ depending on $C_0$ and $Q$ so that 
\begin{equation}\label{eq3.21-0307}
\Mod_Q(\{\gamma\in\Gamma_{j,c}:\ell(\gamma)\le C_1 2^j\})\ge C_0/4.
\end{equation}
Let $\Gamma_{j,c}'$ denote this family of curves, i.e. $\Gamma_{j,c}':=\{\gamma\in\Gamma_{j,c}:\ell(\gamma)\le C_1 2^j\}$.  Each curve in $\Gamma_{j,c}'$ has length at least $2^j$, and so that for all measurable sets $A\supset |\Gamma_{j,c}'|$, $\frac{1}{2^j}\chi_{A}$ is admissible for computing $\Mod_Q(\Gamma_{j,c}')$.  Hence, $C_0/4\le\Mod_Q(\Gamma_{j,c}')\le\mu(|\Gamma_{j,c}'|)/(2^j)^Q$,
and so it follows that 
\begin{equation}\label{eq:volume}
\mu(B(O,2^j))\simeq\mu(|\Gamma_{j,c}'|).
\end{equation}

For each $k\in\N$, let $j_k=k_F+10k$. 
 Choosing a curve $\gamma_{j_k}'$ from $\Gamma_{j_k,c}'$ for each $k\in\mathbb{N}$, we have that 
\[
\lim_{\cup_{k\in\mathbb{N}}\gamma'_{j_k}\ni x\to+\infty}u(x)=0,
\]
and so joining each $\gamma'_{j_k}$ to $\gamma_{j_k}$ and $\gamma_{j_{k+1}}$  yields an infinite curve $\gamma$ such that $\lim_{t\to+\infty}u(\gamma(t))=0$.  We then obtain the desired family $\Gamma^c$ of infinite curves by repeating this process using all of the possible connecting curves in $\Gamma'_{j_k,c}$ for each $k\in\mathbb{N}$.  In doing so, we see from \eqref{eq:volume} that $\mu(|\Gamma^c|\cap B(0,R))\simeq\mu(B(O,R))$ for all sufficiently large $R>20$. Similarly, we have 
$C_0/4\le\Mod_Q(\Gamma_{j,c}')\le\mu(|\Gamma_{j,c}'|\cap A_j)/(2^j)^Q$,
and so it follows that $\mu(B(O,2^j))\simeq\mu(|\Gamma_{j,c}'|\cap A_j)$ and hence $\mu(|\Gamma^c|\cap B(0,R)\setminus B(O,R/2))\simeq\mu(B(O,R))$ for all sufficiently large $R>20$.

Let $\wtil C:=16 \lambda\cdot \bar c_2\cdot C$.  We note that $\wtil C2^j$ is the length of the outer radius of the annulus $G_j$.  Let $R>\wtil C 2^{k_F+20}$, and let $k_0\in\N$ be such that $\wtil C2^{k_F+10k_0}\le R<\wtil C 2^{k_F+10(k_0+1)}$.  Since each $\gamma_{j_k}$ has length comparable to $2^{j_k}$ by Lemma \ref{lem5.2-2308}, as do the connecting curves $\gamma_{j_k}'$, constructed above, we have that  
\[
\ell(\gamma\cap G_{j_k})\simeq 2^{j_k}\simeq\ell(\gamma\cap G'_{j_k})
\]
for all $k\in\N$ and each $\gamma\in\Gamma^c$. Therefore, it follows that 
\begin{align*}
    R\simeq 2^{j_{k_0}}\leq \ell (\gamma\cap  B(O,R)\setminus B(O,R/2))\lesssim \ell(\gamma\cap G_{j_{k_0}}) \le\sum_{k=1}^{k_0+1}\ell(\gamma\cap G_{j_k}')\lesssim\sum_{k=1}^{k_0+1} 2^{j_k}\lesssim 2^{j_{k_0}}\simeq R,
\end{align*}
and hence, each $\gamma\in\Gamma^c$ satisfies  $\ell(\gamma\cap B(O,R))\simeq \ell(\gamma\cap B(O,R)\setminus B(O,R/2))\simeq R$.
This completes the proof.\qedhere

 \end{proof} 

From the preceding lemmas, we now prove Theorem~\ref{thm1.3-2208} and Theorem~\ref{thm4-2705}.

 \begin{proof}
 [Proof of Theorem \ref{thm1.3-2208}]

${\rm I.}\Rightarrow {\rm II.}$ is given by Lemma \ref{thm1.2-0505}.  

${\rm II.}\Rightarrow {\rm III.}$ is given by Lemma \ref{lem4.5-0703}. 


${\rm III.}\Rightarrow {\rm IV.}$ is given by Lemma \ref{lem5.3-0709}.

${\rm IV.}\Rightarrow {\rm I.}$ This is immediate.

Finally, to show uniqueness of the limit, it suffices to prove that there is only one $c\in[-\infty,+\infty]$ for which ${\rm II.}$ holds true.
Suppose that there exist $c,c'\in[-\infty,+\infty]$ and thin sets $E,E'\subset X$ such that $\lim_{E\not\ni x\to+\infty}u(x)=c$ and $\lim_{E'\not\ni x\to+\infty}u(x)=c'$.  Letting $A_j=B(O,2^{j+1})\setminus B(O,2^j)$, $F=X\setminus E$, and $F'=X\setminus E'$, we claim that $F\cap F'\cap A_j\ne\varnothing$ for all sufficiently large $j\in\mathbb N$.  Indeed, 
$E\cup E'$ is a thin set by Lemma~\ref{prop2.5-0605} (4.), and so by Lemma~\ref{prop2.5-0605} (1.), there exists $j_0\in\N$ such that $A_j\setminus (E\cup E')\ne\varnothing$ for all $j\ge j_0$.  Hence, $F\cap F'\cap A_j\ne\varnothing$ for all such $j$, and so we have that $c=c'$.  
 \end{proof}

 \begin{proof}[Proof of Theorem \ref{thm4-2705}]
 We have the following chain of implications:
 \begin{itemize}
     \item $1.\Leftrightarrow 2.\Leftrightarrow 3.$ is given by Theorem~\ref{thm1.3-2208}.
     \item Assume that $2.$ holds.  Then 
	\[
		\lim_{t\to{+\infty}} u(\gamma(t))=c
	\]
	for all $\gamma\in\Gamma^{+\infty}$ with $\#\{j: \gamma\cap E\cap A_{\lambda 2^{j+1}}\neq\emptyset\}<{+\infty}$.
 Roughly speaking, for all infinite curves $\gamma$ that do not meet $E$ at infinity, we have the limit at infinity along such $\gamma$ exists and is equal to $c$.
{
 Moreover, 
	by the third claim  of Lemma \ref{prop2.5-0605}, we have 
	\[
		\lim_{k\to{+\infty}}\sum_{j=k}^{+\infty}\mathcal H^\alpha_{+\infty, d_{\mathbb S}}(p(\Gamma^{\mathcal O}(\mathbb S')\cap E\cap A_{\lambda 2^{j+1}}))=0 \text{\rm \ \ for all $0<\alpha<Q$}.
	\]  Let $E^*:= \cap_{k\in\mathbb N}\cup_{j=k}^{+\infty} p( \Gamma^{\mathcal O}(\mathbb S')\cap E\bigcap A_{\lambda 2^{j+1}})$. This estimate yields $\mathcal H^\alpha_{+\infty, d_{\mathbb S}}(E^*)=0$ for all $0<\alpha<Q$, and hence $\mathcal H_{d_{\mathbb S}}^\alpha(E^*)=0$ for all $0<\alpha<Q$ by Lemma \ref{lem2.1-0805}. 
	}
	It is clear that $\lim_{t\to{+\infty}}u(\gamma^{\mathcal O}_\xi(t))=c$ for {
	$\xi\in\mathbb S'\setminus E^*$}.  Thus $4.$ holds, and so $2.\Rightarrow 4.$  
    \item $4.\Rightarrow 5.$ is immediate.
    \item {
    $5.\Rightarrow 6.$ follows from the assumption that $\sigma\ll \Ha^{\alpha_0}_{d_{\mathbb S}}$.}
    \item {
    $6.\Rightarrow 1.$ follows since $\Gamma^\mathcal{O}(\mathbb{S}')$ is nonempty by $\sigma(\mathbb S')>0$.}
    \item Finally, the uniqueness of the limit $c$ follows from that same argument as in Theorem~\ref{thm1.3-2208}.\qedhere
 \end{itemize}
	



 \end{proof}

\section{Proof of Theorem \ref{thm1-2705} and Theorem~\ref{thm:p>Q Growth}}
Recall that the scaling constant of the Poincar\'e inequality is denoted by $\lambda$, and we define $A_{\lambda 2^{j+1}}:=B(O,\lambda 2^{j+2})\setminus B(O,\lambda 2^{j+1})$.  For a constant $C\ge 1$, we define $C\cdot A_{\lambda 2^{j+1}}:=B(O,C\lambda 2^{j+3})\setminus B(O,2^{j}/(C\lambda))$.


\begin{lemma}\label{thm1-0705}
Let  $1<Q<+\infty$. 
Suppose that $(X,d,\mu)$ is a complete, unbounded metric measure space with metric $d$ and Ahlfors $Q$-regular measure $\mu$ supporting a $Q$-Poincar\'e inequality. 
 Then for every function $u\in \dot N^{1,Q}(X)$, there is a thin set $E$ such that
 	\[
		\lim_{E\not\ni x\to{+\infty}}\left |u(x) -u_{B(x, |x|/2)}\right |=0.
	\]
 \end{lemma}
 
 \begin{proof}
 Let $+\infty>C\ge 1$ be chosen so that for all $j\in\N$ and $x\in A_{\lambda 2^{j+1}}$, we have that $B(x,|x|/2)\subset C\cdot A_{\lambda 2^{j+1}}$.  Since $u\in \dot N^{1,Q}(X)$ and by \eqref{eq3.1-2805} in Lemma \ref{in-fact}, there is a non-increasing sequence $\{a_j\}_{j\in\mathbb N}$ with $a_j>0$ such that
 	\[
		\lim_{j\to{+\infty}}a_j=0\text{\rm \ \ and\ \ }  \sum_{j=1}^{+\infty} a_{j}^{-Q}\int_{C\cdot A_{\lambda 2^{j+1}}}\rho_u^Qd\mu<{+\infty}.
	\]
	 
	Let $E_j:=\{x\in A_{\lambda 2^{j+1}}:|u(x)-u_{B(x,|x|/2)}|>a_j\}$. Let $0<\alpha<Q$. By our choice of $C$, Lemma \ref{lem3.2} yields
	\[
		\frac{\mathcal H^\alpha_{+\infty}(E_j)}{(2^j)^\alpha }\lesssim \frac{1}{(1-2^{-\frac{\alpha}{Q}})^Q} a_j^{-Q}\int_{C\cdot A_{\lambda 2^{j+1}}}\rho_u^Qd\mu.
	\]
	  Let $E:=\bigcup_{j=1}^{+\infty} E_j$. Combining the above estimates, we have that
	\[
	 \lim_{k\to{+\infty}}\sum_{j=k}^{+\infty}
	 \frac{\mathcal H^\alpha_{+\infty}(E_j)}{(\lambda 2^{j+1})^\alpha }\lesssim \frac{1}{(1-2^{-\frac{\alpha}{Q}})^Q} \lim_{k\to+\infty}\sum_{j=k}^{+\infty}a_j^{-Q}\int_{C\cdot A_{\lambda 2^{j+1}}}\rho_u^Qd\mu=0.
	\] Then $E$ is a thin set, and by our construction of $E$, we have that  $\lim_{E\not\ni x\to{+\infty}}|u(x)-u_{B(x,|x|/2)}|=0$, proving the claim.
 \end{proof}
 \begin{lemma}\label{thm3.3-2605}
Let  $1<Q<+\infty$. 
Suppose that $(X,d,\mu)$ is a complete, unbounded metric measure space with metric $d$ and Ahlfors $Q$-regular measure $\mu$ supporting a $Q$-Poincar\'e inequality.
Then for every function $u\in\dot N^{1,Q}(X)$, there is a thin set $E$ such that
 \begin{equation}\label{eq3.2-2805}
 \lim_{E\not\ni x\to{+\infty}}\frac{|u(x)-u_{B(x,1)}|}{\log^{\frac{Q-1}{Q}}(|x|)}=0.
 \end{equation}
 \end{lemma}

 \begin{proof}
	 Since $u\in \dot N^{1,Q}(X)$ and by Lemma \ref{in-fact}, there is a sequence $\{a_j\}_{j\in\mathbb N}$ with $a_j>0$ such that
 	\begin{equation}\label{eq3.1-1305}
 	\lim_{j\to{+\infty}}a_j=0 \text{\rm \ \ and \ \ } \sum_{j=1}^{+\infty} a_{j}^{-Q}\int_{3\cdot A_{\lambda 2^{j+1}}}\rho_u^Qd\mu<{+\infty}.
 	\end{equation}
	For each $j\in\N$, let 
	\[
			E_j:=\left\{x\in A_{\lambda 2^{j+1}}:\frac{|u(x)-u_{B(x,1)}|
	}{\log^{\frac{Q-1}{Q}}(|x|)}>a_j\right\}.
	\]	
	It is clear that \eqref{eq3.2-2805} holds with $E:=\bigcup_{j=1}^{+\infty} E_j$, and so it remains to show that $E$ is a thin set. 
	For each $x\in E_j$, we have that
	\begin{align*}
	a_j<\frac{|u(x)-u_{B(x,1)}|}{\log^{\frac{Q-1}{Q}}(|x|)}\leq& |u(x)-u_{B(x,|x|/2)}| + \frac{|u_{B(x,|x|/2)}-u_{B(x,1)}|}{\log^{\frac{Q-1}{Q}}(|x|)}.
	\end{align*}
	Then $E_j\subset E_j^1\bigcup E_j^2$, where  $E_j^1:=\{x\in A_{\lambda 2^{j+1}}: |u(x)-u_{B(x,|x|/2)}|\geq a_j/2\}$ and
 \[
 E_j^2:=\left\{x\in A_{\lambda 2^{j+1}}: \frac{|u_{B(x,|x|/2)}{-u_{B(x,1)}}|}{\log^{\frac{Q-1}{Q}}(|x|)}\geq a_j/2 \right\}.
 \] 
 By the proof of Lemma~\ref{thm1-0705}, we have that $E^1:=\bigcup_{j=1}^{+\infty} E_j^1$ is a thin set.
 
By the claim ($4.$) of Lemma~\ref{prop2.5-0605}, it suffices to prove that $E^2:=\bigcup_{j=1}^{+\infty} E_j^2$ is a thin set. For each $x\in X$, let $k_x\in \mathbb N$ be such that $1\leq 2^{-k_x} |x|/2\leq 2$, and so $k_x+1 \lesssim \log(|x|)$ if $|x|$ is sufficiently large. For $x\in E_j^2$, we have 
	\begin{align*}
	\frac{a_j}{2}\leq \frac{|u_{B(x,|x|/2)}{-u_{B(x,1)}}|}{\log^{\frac{Q-1}{Q}}(|x|)}
	 &\leq\frac{1}{\log^{\frac{Q-1}{Q}}(|x|)}\left(|u_{B(x,1)}-u_{B(x,2^{-(k_x+1)}|x|)}|+\sum_{i=1}^{k_x}|u_{B(x,2^{-i}|x|)}-u_{B(x,2^{-i-1}|x|)}|\right)\\
    &\lesssim\frac{1}{\log^{\frac{Q-1}{Q}}(|x|)}\sum_{i=1}^{k_x+1}\left(\int_{\lambda B(x,2^{-i}|x|)}\rho_u^Q\,d\mu\right)^{1/Q}\\  
    &\le\frac{1}{\log^{\frac{Q-1}{Q}}(|x|)}\left(\sum_{i=1}^{k_x+1}1\right)^{\frac{Q-1}{Q}}\left(\sum_{i=1}^{k_x+1}\int_{\lambda B(x,2^{-i}|x|)}\rho_u^Q\,d\mu\right)^{1/Q}\\
    &\lesssim\left(\sum_{i=1}^{k_x+1}\int_{\lambda B(x,2^{-i}|x|)}\rho_u^Q\,d\mu\right)^{1/Q}.
    \end{align*}
Here we have used the $Q$-Poincar\'e inequality along with Ahlfors $Q$-regularity, H\"older's inequality, and the fact that $k_x+1\lesssim\log(|x|)$. Let $0<\alpha<Q$. Again by Ahlfors $Q$-regularity, we have that for 
all $x\in E^2_j$, 
\begin{align*}
a_j^Q\lesssim\sum_{i=1}^{k_x+1}(2^{-i}|x|)^\alpha(2^{-i}|x|)^{Q-\alpha}\dashint_{\lambda B(x,2^{-i}|x|)}\rho_u^Q\,d\mu&\le\mathcal M_{Q-\alpha,\lambda2^{j+1}}\rho_u^Q(x)\sum_{i=1}^{k_x+1}(2^{-i}|x|)^\alpha\\
    &\lesssim \frac{1}{1-2^{-\alpha}} (\lambda 2^{j+1})^\alpha\mathcal M_{Q-\alpha,\lambda2^{j+1}}\rho_u^Q(x),
\end{align*}
where the last inequality holds since $|x|\le \lambda 2^{j+1}$.
By applying Theorem 
\ref{fractional-maximal-theorem} applied to $\rho_u\chi_{3\cdot A_{\lambda 2^{j+1}}}$, the above estimate yields
	\[
		\mathcal H^{\alpha}_{+\infty}(E_j^2)\lesssim  \frac{1}{1-2^{-\alpha}} \frac{(\lambda 2^{j+1})^\alpha}{a_j^Q}\int_{3\cdot A_{\lambda 2^{j+1}}}\rho_u^Qd\mu.
	\]
From \eqref{eq3.1-1305}, we now obtain  
\[
\lim_{k\to{+\infty}}\sum_{j=k}^{+\infty} 
\frac{\mathcal H^\alpha_{+\infty}(E_j^2)}{(\lambda 2^{j+1})^\alpha}\lesssim  \frac{1}{1-2^{-\alpha}}  \lim_{k\to{+\infty}}\sum_{j=k}^{+\infty} \frac{1}{a_j^Q}\int_{3\cdot A_{\lambda 2^{j+1}}}\rho_u^Qd\mu=0,
\]	
and so $E^2$ is a thin set, completing the proof.
\end{proof}

Recall that a metric space $(X,d)$ is called geodesic if every pair of points in $(X,d)$  can be joined by a curve whose length is the distance between the points.

\begin{lemma}\label{lem3.4-2805}
Let  $1<Q<+\infty$. 
Suppose that $(X,d,\mu)$ is an  unbounded geodesic metric measure space with metric $d$ and Ahlfors $Q$-regular measure $\mu$ supporting a $Q$-Poincar\'e inequality. We assume that $X$ is complete, and let $c_2\ge 1$ and $\overline {c_2}\ge 1$ be as in Definition~\ref{dyadic} and Proposition~\ref{prop2.3-0806}, respectively.
 Then there exists a finite constant $C_1\ge 1$ such that for all $u\in\dot N^{1,Q}(X)$ and for all $x_0,x_1\in X$ satisfying $|x_0|>8c_2$ and $|x_1|/|x_0|\ge |x_0|$, the following holds:

 \begin{equation}\label{eq3.4-2805 SECOND ATTEMPT}
	|u_{B(x_0,1)}-u_{B(x_1,1)}| \leq C_1\left(\log \left( \frac{|x_1|}{|x_0|}\right)\right)^{\frac{Q-1}{Q}}  \|\rho_u\|_{L^Q_\mu\left(X\setminus B\left(O,\frac{|x_0|}{2{\overline{c_2}c_2}} \right)\right)}.
	\end{equation}
\end{lemma}

\begin{proof}
    Let $k\in\N\cup\{0\}$ be such that 
    \[
    2^k|x_0|\le\frac{|x_1|}{|x_0|}<2^{k+1}|x_0|.
    \]
    For each $0\le i\le k+2$, choose $a_i\in X$ such that $d(O,a_i)=2^i|x_0|$.  This is possible since $X$ is connected as a consequence of the $Q$-Poincar\'e inequality.  

    By \cite[Theorem~3.3]{K07} and Proposition \ref{prop2.3}, we see that $X$ has the annular chain property.  Let $\sigma>4$ be defined by
    \[
   \sigma:=100\left(\frac{c_2}{c_1}+1 \right)
    \]
    where $c_1$ and $c_2$ are the constants given by Definition~\ref{dyadic}. 
    Since $X$ has the annular chain property, we can join the balls $B(a_{i-1},2^i|x_0|/(\sigma c_1))$ and $B(a_{i},2^i|x_0|/(\sigma c_1))$, for $1\le i\le k+2$, by a finite chain of balls $\{B_{i,j}\}_{j=1}^M$ as given by Definition~\ref{dyadic}.  In particular, each $B_{i,j}$ has radius $2^i|x_0|/(\sigma c_1)$ and $B_{i,j}\subset B(O,c_2 2^i|x_0|)\setminus B(O,2^i|x_0|/c_2)$, where $M$, $c_1$, and $c_2$ are the constants given in Definition~\ref{dyadic}, with $M$ depending on $c_1$ and $c_2$.  While each such chain contains at most $M$ balls, a chain may contain strictly fewer than $M$ balls.  However, for simplicity, we still index the chain by $1\le j\le M$, repeating the final ball in the chain if necessary. By the triangle inequality, we have that 
    \begin{align}\label{eq:2406-1}
        |u_{B(x_0,1)}-u_{B(x_1,1)}|\le|u_{B(x_0,1)}-&u_{B(a_0,2|x_0|/(\sigma c_1))}|+\sum_{i=1}^{k+2}\sum_{j=1}^{M-1}\left(|u_{B_{i,j}}-u_{B_{i,j+1}}|\right)+\sum_{i=1}^{k+1}\left(|u_{B_{i,M}}-u_{B_{i+1,1}}|\right)\nonumber\\
        &+|u_{B(a_{k+2},2^{k+2}|x_0|/(\sigma c_1))}-u_{B(x_1,1)}|,
    \end{align}
    and by the $Q$-Poincar\'e inequality, Ahlfors $Q$-regularity, and H\"older's inequality, it follows that 
    \begin{align*}
        \sum_{i=1}^{k+2}\sum_{j=1}^{M-1}(|u_{B_{i,j}}-u_{B_{i,j+1}}|)&+\sum_{i=1}^{k+1}(|u_{B_{i,M}}-u_{B_{i+1,1}}|)\\
        &\lesssim\sum_{i=1}^{k+2}\sum_{j=1}^{M}\left(\left(\int_{B_{i,j}}\rho_u^Qd\mu\right)^{1/Q}+\left(\int_{2 B_{i,M}}\rho_u^Qd\mu\right)^{1/Q}\right)\\
        &\le\sum_{i=1}^{k+2} M^{\frac{Q-1}{Q}}\left(\sum_{j=1}^{M}\int_{ B_{i,j}}\rho_u^Qd\mu\right)^{1/Q}+M(k+2)^{\frac{Q-1}{Q}}\left(\sum_{i=1}^{k+2}\int_{2 B_{i,M}}\rho_u^Qd\mu\right)^{1/Q}.
    \end{align*}
    Note that we may assume the scaling constant of the Poincar\'e inequality to be 1, by the Ahlfors $Q$-regularity of $\mu$ and the assumption that $X$ is a geodesic space, see \cite[Theorem~9.5]{Hei01}.
    
    Since $B_{i,j}\subset B(O,c_2 2^i|x_0|)\setminus B(O,2^i|x_0|/c_2)=:A_i$, for $1\le i\le k+2$, it follows 
    from H\"older's inequality, the bounded overlap of the $A_i$, and the choice of $k$ that
    \begin{align}\label{eq:2806-3}
        \sum_{i=1}^{k+2}\sum_{j=1}^{M-1}(|u_{B_{i,j}}-&u_{B_{i,j+1}}|)+\sum_{i=1}^{k+1}(|u_{B_{i,M}}-u_{B_{i+1,1}}|)\nonumber\\
        &\lesssim M\sum_{i=1}^{k+2}\left(\int_{A_i}\rho_u^Qd\mu\right)^{1/Q}+M(k+2)^{\frac{Q-1}{Q}}\left(\sum_{i=1}^{k+2}\int_{2B_{i,M}}\rho_u^Qd\mu\right)^{1/Q}\nonumber\\
        &\lesssim M(k+2)^{\frac{Q-1}{Q}}\|\rho_u\|_{L^Q_\mu(X\setminus B(O,|x_0|/c_2))}\lesssim M\log^{\frac{Q-1}{Q}}\left(\frac{|x_1|}{|x_0|}\right)\|\rho_u\|_{L^Q_\mu(X\setminus B(O,|x_0|/c_2))}.
    \end{align}

    We now use a similar argument to estimate the first term on the right hand side of \eqref{eq:2406-1}.  Let $k_0\in\N$ be such that 
    \[
    2^{k_0}\le \frac{2|x_0|}{\sigma c_1}<2^{k_0+1}.
    \]
    For each $0\le i\le k_0$, let $b_i\in X$ be such that $d(x_0,b_i)=2^i$.  Again, 
 the connectedness of $X$ makes this possible. Similarly as above, for each $1\le i\le k_0$, we can join the balls $B(b_{i-1},2^i/(\sigma c_1))$ and $B(b_i,2^i/(\sigma c_1))$ with a finite chain of balls $\{B_{i,j}\}_{j=1}^M$ as in Definition~\ref{dyadic}.  That is, each $B_{i,j}$ has radius $2^i/(\sigma c_1)$ and $B_{i,j}\subset B(x_0,c_2 2^i)\setminus B(x_0,2^i/c_2):=A_i(x_0)$.  
 By the $Q$-Poincar\'e inequality, Ahlfors $Q$-regularity, and H\"older's inequality, it then follows that 
 \begin{align}\label{eq:2806-1}
     &|u_{B(x_0,1)}-u_{B(a_0,2|x_0|/(\sigma c_1))}|\nonumber\\
     \le&|u_{B(x_0,1)}-u_{B(b_0,1/(\sigma c_1))}|\nonumber+\sum_{i=1}^{k_0}|u_{B(b_{i-1},2^i/(\sigma c_1))}-u_{B(b_i,2^i/(\sigma c_1))}|\\
     &+\sum_{i=1}^{k_0}|u_{B(b_{i-1}, 2^{i-1}/(\sigma c_1))}-u_{B(b_{i-1}, 2^i/(\sigma c_1))}|+|u_{B(b_{k_0},2^{k_0}/(\sigma c_1))}-u_{B(x_0,2|x_0|/(\sigma c_1))}|\nonumber\\
     \lesssim&\left(\int_{B(x_0,2)}\rho_u^Qd\mu\right)^{1/Q}+2\sum_{i=1}^{k_0}\sum_{j=1}^M\left(\int_{B_{i,j}}\rho_u^Qd\mu\right)^{1/Q}+\left(\int_{B(x_0,4|x_0|/(\sigma c_1))}\rho_u^Qd\mu\right)^{1/Q}\nonumber\\
     \lesssim&\|\rho_u\|_{L^Q_\mu(X\setminus B(O,|x_0|/(2 c_2)))}+Mk_0^{\frac{Q-1}{Q}}\left(\sum_{i=1}^{k_0}\int_{A_i(x_0)}\rho_u^Qd\mu\right)^{1/Q}\nonumber\\
    \lesssim&\|\rho_u\|_{L^Q_\mu(X\setminus B(O,|x_0|/(2 c_2)))}+M\log^{\frac{Q-1}{Q}}(|x_0|)\|\rho_u\|_{L^Q_\mu(B(x_0,2 c_2|x_0|/(\sigma c_1)))}\nonumber\\
    \lesssim& M\log^{\frac{Q-1}{Q}}\left(\frac{|x_1|}{|x_0|}\right)\|\rho_u\|_{L^Q_\mu(X\setminus B(O,|x_0|/(2 c_2)))}.
 \end{align}
 Here we have used the choice of $\sigma$ to ensure that $B(x_0,4|x_0|/(\sigma c_1))$ and $B(x_0, 2 c_2|x_0|/(\sigma c_1))$ are subsets of $X\setminus B(O, |x_0|/(2 c_2))$.  We have also used the assumption that $|x_1|/|x_0|\ge |x_0|$ and $|x_0|> 8 c_2$ to ensure that $k_0\lesssim\log|x_0|\lesssim\log(|x_1|/|x_0|)$.

To estimate the last term on the right hand side of \eqref{eq:2406-1}, we note that the ball $B(a_{k+2},\frac{|x_1|}{|x_0|})$ and the ball $B(x_1,\frac{|x_1|}{|x_0|})$ are both contained in $B(O,2\overline{c_1}|x_1|)\setminus B(O,|x_0|/\overline{c_1})$.  Proposition~\ref{prop2.3-0806} then gives us a chain of balls $\{B_i\}_{i=1}^{\overline{M}}$ joining $B(a_{k+2},\frac{|x_1|}{|x_0|})$ and $B(x_1,\frac{|x_1|}{|x_0|})$ so that each $B_i$ has radius comparable to $|x_1|/|x_0|$ and is contained in $B(O,2\overline{c_2}|x_1|)\setminus B(O,|x_0|/\overline{c_2})$.  Here, $\overline{c_1}$, $\overline{c_2}$, and $\overline M$ are the constants from Proposition~\ref{prop2.3-0806}.  Thus, by the $Q$-Poincar\'e inequality, Ahlfors $Q$-regularity, H\"older's inequality, and the choice of $k$, we have that 
\begin{align*}
    |u_{B(a_{k+2},2^{k+2}|x_0|/(\sigma c_1))}-&u_{B(x_1,1)}|\le|u_{B(a_{k+2},2^{k+2}|x_0|/(\sigma c_1))}-u_{B(a_{k+2},\frac{|x_1|}{|x_0|})}|\\
    &+\sum_{i=1}^{\overline{M}-1}|u_{B_i}-u_{B_{i+1}}|+|u_{B(x_1,\frac{|x_1|}{|x_0|})}-u_{B(x_1,1)}|\\
&\lesssim\left(\int_{B(a_{k+2},\frac{|x_1|}{|x_0|})}\rho_u^Qd\mu\right)^{1/Q}+\sum_{i=1}^{\overline M-1}\left(\int_{B_i}\rho_u^Qd\mu\right)^{1/Q}+|u_{B(x_1,\frac{|x_1|}{|x_0|})}-u_{B(x_1,1)}|\\
&\lesssim \overline M\|\rho_u\|_{L^Q_\mu(X\setminus B(O,|x_0|/\overline{c_2}))}+|u_{B(x_1,\frac{|x_1|}{|x_0|})}-u_{B(x_1,1)}|.
\end{align*}
Here we have used the choice of $\sigma>4$ to ensure that $B(a_{k+2},2^{k+2}|x_0|/(\sigma c_1))\subset B(a_{k+2},|x_1|/|x_0|)$.
By the same chaining argument used to obtain estimate \eqref{eq:2806-1}, we have that 
\[
|u_{B(x_1,\frac{|x_1|}{|x_0|})}-u_{B(x_1,1)}|\lesssim M\log^{\frac{Q-1}{Q}}\left(\frac{|x_1|}{|x_0|}\right)\|\rho_u\|_{L^Q_\mu(X\setminus B(O,|x_0|/(2c_2)))},
\]
and so combining this with the previous estimate, we obtain
\begin{equation}\label{eq:2806-2}
    |u_{B(a_{k+2},2^{k+2}|x_0|/(\sigma c_1))}-u_{B(x_1,1)}|\lesssim(M+\overline M)\log^{\frac{Q-1}{Q}}\left(\frac{|x_1|}{|x_0|}\right)\|\rho_u\|_{L^Q_\mu(X\setminus B(O,|x_0|/(2\overline{c_2}c_2)))}.
\end{equation}
Therefore, combining \eqref{eq:2806-3}, \eqref{eq:2806-1}, and \eqref{eq:2806-2} with \eqref{eq:2406-1} completes the proof.
\end{proof}

 \begin{proof}[Proof of Theorem \ref{thm1-2705}] 
 The first  claim \eqref{eq1.1-2705} follows from the second claim \eqref{equ1.2-2805} and Lemma~\ref{thm3.3-2605}, and so it suffices to prove \eqref{equ1.2-2805}.

 Let $u\in\dot N^{1,Q}(X)$.  Since $(X,d,\mu)$ is complete, Ahlfors $Q$-regular, and supports a $Q$-Poincar\'e inequality, it follows from \cite[Corollary~8.3.16]{pekka} that $X$ admits a geodesic metric $d'$ which is bi-Lipschitz equivalent to the original metric $d$.  Then $(X,d',\mu)$ is also Ahlfors $Q$-regular and supports a $Q$-Poincar\'e inequality.  For each $x\in X$ and $r>0$, we  denote $B'(x,r)$  the ball centered at $x$ of radius $r$ with respect to $d'$. 

 By the triangle inequality, we have that 
 \[
 \left|\frac{u_{B(x,1)}}{\log^{\frac{Q-1}{Q}}(|x|)}\right|\le\left|\frac{u_{B(x,1)}-u_{B'(x,1)}}{\log^{\frac{Q-1}{Q}}(|x|)}\right|+\left|\frac{u_{B'(x,1)}}{\log^{\frac{Q-1}{Q}}(|x|)}\right|.
 \]
 By the $Q$-Poincar\'e inequality, Ahlfors $Q$-regularity, and the fact that $d$ is bi-Lipschitz equivalent to $d'$, it follows that 
 \begin{align*}
     |u_{B(x,1)}-u_{B'(x,1)}|\lesssim\left(\int_{B(x,C)}\rho_u^Qd\mu\right)^{1/Q}\to 0
 \end{align*}
 as $x\to+\infty$, since $\rho_u\in L^Q_\mu(X)$.  Here the constant $C$ depends on the bi-Lipschitz constant. We let $|x|':=d'(O,x)$.  As such it suffices to show that 
 \[
 \lim_{x\to+\infty}\left|\frac{u_{B'(x,1)}}{\log^{\frac{Q-1}{Q}}(|x|')}\right|=0.
 \]
 	
  We argue by  contradiction. Assume that the above does not hold.  Then there is a constant $\delta>0$ such that
	\begin{equation}\label{eq3.5-2805}
		\limsup_{x\to{+\infty}}\left| \frac{u_{B'(x,1)}}{\log^{\frac{Q-1}{Q}}(|x|')} \right|>2\delta.
	\end{equation}
	Since $d$ is bi-Lipschitz equivalent to $d'$, it follows that  $u\in\dot N^{1,Q}(X,d',\mu)$, and so we can choose $x_0\in X$, with $|x_0|'$ sufficiently large, such that $|u_{B'(x_0,1)}|<\infty$, $|x_0|'> 8c_2$, and 
	\begin{equation}\label{eq-delta-2605}
		\delta > 2 C_{1} \|\rho_u\|_{L^Q_\mu\left(X\setminus B'\left(O,\frac{|x_0|'}{2\overline{c_2}c_2}\right)\right)},
	\end{equation}
	where $C_1$, $\overline{c_2}$, and $c_2$ are the constants in Lemma \ref{lem3.4-2805}.
	Likewise, from \eqref{eq3.5-2805}, we can choose $x_1\in X$, with $|x_1|'$ sufficiently large so that $|x_1|'/|x_0|'\ge|x_0|'$ and 
	\begin{equation}\label{eq3-2605}
		|u_{B'(x_1,1)}|>2\delta \log^{\frac{Q-1}{Q}}(|x_1|')>2 |u_{B'(x_0,1)}|.
	\end{equation}
From this estimate, it then follows that
	\begin{align*}
		\delta \log^{\frac{Q-1}{Q}}\left(\frac{|x_1|'}{|x_0|'} \right) \leq  \delta \log^{\frac{Q-1}{Q}}(|x_1|')\leq \frac{1}{2}|u_{B'(x_1,1)}|&=|u_{B'(x_1,1)}|-|u_{B'(x_0,1)}|+ \frac{2|u_{B'(x_0,1)}|-|u_{B'(x_1,1)}|}{2}\\
		&\leq  |u_{B'(x_1,1)}- u_{B'(x_0,1)}|
	\end{align*}
where the first inequality is given because $|x_0|'\ge 8$. Combining this estimate with Lemma \ref{lem3.4-2805}, we obtain 
	\[
		\delta \log^{\frac{Q-1}{Q}}\left(\frac{|x_1|'}{|x_0|'} \right) \leq  C_1 \log^{\frac{Q-1}{Q}}\left(\frac{|x_1|'}{|x_0|'} \right)  \|\rho_u\|_{L^Q_\mu\left(X\setminus B'\left(O,\frac{|x_0|'}{2\overline{c_2}c_2} \right)\right)},
	\]
	which implies that
	\[
	\delta \leq  C_1  \|\rho_u\|_{L^Q_\mu\left(X\setminus B'\left(O,\frac{|x_0|'}{2\overline{c_2}c_2} \right)\right)}.
	\]
This contradicts \eqref{eq-delta-2605}, completing the proof. 
 \end{proof}

Recall that when $p>Q$, 
the assumption that   $(X,d,\mu)$ is a metric measure space with metric $d$ and Ahlfors $Q$-regular measure $\mu$ supporting a $p$-Poincar\'e inequality,
 does not guarantee that $X$ satisfies the annular chain property.  Therefore to prove Theorem~\ref{thm:p>Q Growth}, we must make the further assumption that this property holds.  We now use Theorem~\ref{thm:Morrey} to prove the following analog of Lemma~\ref{lem3.4-2805} for the case $p>Q$:
\begin{lemma}\label{lem:p>Q LogLemma}
Let $1<Q<+\infty$, and let $Q<p<+\infty$.  Suppose that $(X,d,\mu)$ is a complete, unbounded geodesic metric measure space with metric $d$ and Ahlfors $Q$-regular measure $\mu$ supporting a $p$-Poincar\'e inequality.  Furthermore, we assume that $X$ satisfies the annular chain property.  Then there exists a constant $C_1\ge 1$ such that for all $u\in\dot N^{1,p}(X)$ and for all $x_0,x_1\in X$ satisfying $|x_0|>8c_2$ and $|x_1|/|x_0|\ge |x_0|$, the following holds:
\begin{equation}
    |u_{B(x_0,1)}-u_{B(x_1,1)}|\le C_1\left(\frac{|x_1|}{|x_0|}\right)^{1-Q/p}\|\rho_u\|_{L^p_\mu\left(X\setminus B\left(O,\frac{|x_0|}{2\overline c_2c_2}\right)\right)}.
\end{equation}
Here $c_2$ and $\overline c_2$ are the constants given by Definition~\ref{dyadic} and Proposition~\ref{prop2.3-0806}, respectively.    
\end{lemma}

\begin{proof}
    As in the proof of Lemma~\ref{lem3.4-2805}, we let $k\in\N$ be such that 
    \[
    2^k|x_0|\le\frac{|x_1|}{|x_0|}<2^{k+1}|x_0|,
    \]
and for each $0\le i\le k+2$, we choose $a_i\in X$ such that $d(O,a_i)=2^i|x_0|$.  Letting 
\[
\sigma:=100\left( \frac{c_2}{c_1}+1\right)
\]
where $c_1$ and $c_2$ are the constants from Definition~\ref{dyadic}, we can join the balls $B(a_{i-1},2^i|x_0|/(\sigma c_1))$ and $B(a_i,2^i|x_0|/(\sigma c_1))$, for $1\le i\le k+2$ by a finite chain of balls $\{B_{i,j}\}_{j=1}^M$ where $M$ depends on $c_1$ and $c_2$.  In particular, each $B_{i,j}$ has radius $2^i|x_0|/(\sigma c_1)$ and is contained in $B(O,c_22^i|x_0|)\setminus B(O, 2^i|x_0|/c_2)$.  As in the proof of Lemma~\ref{lem3.4-2805}, we have that 
 \begin{align}\label{eq:3008-1}
        |u_{B(x_0,1)}-u_{B(x_1,1)}|\le|u_{B(x_0,1)}-&u_{B(a_0,2|x_0|/(\sigma c_1))}|+\sum_{i=1}^{k+2}\sum_{j=1}^{M-1}\left(|u_{B_{i,j}}-u_{B_{i,j+1}}|\right)+\sum_{i=1}^{k+1}\left(|u_{B_{i,M}}-u_{B_{i+1,1}}|\right)\nonumber\\
        &+|u_{B(a_{k+2},2^{k+2}|x_0|/(\sigma c_1))}-u_{B(x_1,1)}|.
    \end{align}
    
    By Theorem~\ref{thm:Morrey} and Ahlfors $Q$-regularity, we have that
    \begin{align}\label{eq:Morrey piece 1}
    |u_{B(x_0,1)}-&u_{B(a_0,2|x_0|/(\sigma c_1))}|\nonumber\\
    &\lesssim|x_0|^{1-Q/p}\left(\int_{B(x_0,8|x_0|/(\sigma c_1)}\rho_u^pd\mu\right)^{1/p}\le\left(\frac{|x_1|}{|x_0|}\right)^{1-Q/p}\|\rho_u\|_{L^p_\mu\left(X\setminus B\left(O,\frac{|x_0|}{2c_2}\right)\right)}.
    \end{align}
    Here we have used the choice of $\sigma$ to ensure that $B(x_0,8|x_0|/(\sigma c_1))\subset X\setminus B(O,|x_0|/(2c_2))$.  We note that the scaling factor of the Poincar\'e inequality is $1$, since $X$ is assumed to be geodesic.  
    
    To estimate the second term on the right hand side of \eqref{eq:3008-1}, we obtain from the Ahflors $Q$-regular property, the $p$-Poincar\'e inequality, and the H\"older's inequality that
    \begin{align*}
        \sum_{i=1}^{k+2}\sum_{j=1}^{M-1}&(|u_{B_{i,j}}-u_{B_{i,j+1}}|)+\sum_{i=1}^{k+1}(|u_{B_{i,M}}-u_{B_{i+1,1}}|)\\
        &\lesssim\sum_{i=1}^{k+2}\sum_{j=1}^{M}\left((2^i|x_0|)^{1-Q/p}\left(\int_{B_{i,j}}\rho_u^pd\mu\right)^{1/p}+(2^i|x_0|)^{1-Q/p}\left(\int_{2 B_{i,M}}\rho_u^pd\mu\right)^{1/p}\right)\\
        &\le\sum_{i=1}^{k+2} (2^i|x_0|)^{1-Q/p}\sum_{j=1}^{M}\left(\int_{ B_{i,j}}\rho_u^pd\mu\right)^{1/p}+M\sum_{i=1}^{k+2}(2^i|x_0|)^{1-Q/p}\left(\int_{2 B_{i,M}}\rho_u^pd\mu\right)^{1/p}\\
        &\lesssim M^{(p-1)/p}\sum_{i=1}^{k+2}(2^i|x_0|)^{1-Q/p}\left(\sum_{j=1}^{M}\int_{B_{i,j}}\rho_u^pd\mu\right)^{1/p}+M(2^k|x_0|)^{1-Q/p}\left(\sum_{i=1}^{k+2}\int_{2B_{i,M}}\rho_u^pd\mu\right)^{1/p}.
    \end{align*}
Since $B_{i,j}\subset B(O,c_22^i|x_0|)\setminus B(O,2^i|x_0|/c_2)=:A_i$, we have by the bounded overlap of the $A_i$, H\"older's inequality, and choice of $k$ that 
\begin{align}\label{eq:3008-2}
 \sum_{i=1}^{k+2}\sum_{j=1}^{M-1}(|u_{B_{i,j}}&-u_{B_{i,j+1}}|)+\sum_{i=1}^{k+1}(|u_{B_{i,M}}-u_{B_{i+1,1}}|)\nonumber\\   
  &\lesssim M\sum_{i=1}^{k+2}(2^i|x_0|)^{1-Q/p}\left(\int_{A_i}\rho_u^pd\mu\right)^{1/p}+M(2^k|x_0|)^{1-Q/p}\left(\sum_{i=1}^{k+2}\int_{2B_{i,M}}\rho_u^pd\mu\right)^{1/p}\nonumber\\
  &\lesssim M (2^k|x_0|)^{1-Q/p}\|\rho_u\|_{L^p_{\mu}\left(X\setminus B\left(O,\frac{|x_0|}{2c_2}\right)\right)}\lesssim M\left(\frac{|x_1|}{|x_0|}\right)^{1-Q/p}\|\rho_u\|_{L^p_{\mu}\left(X\setminus B\left(O,\frac{|x_0|}{2c_2}\right)\right)}.
\end{align}

To estimate the last term on the right hand side of \eqref{eq:3008-1}, we use Proposition~\ref{prop2.3-0806} to join $B(a_{k+2},|x_1|/|x_0|)$ and $B(x_1,|x_1|/|x_0|)$ with the same chain of balls $\{B_i\}_{i=1}^{\overline M}$ used in the proof of Lemma~\ref{lem3.4-2805}.  Similarly as in that proof, we use Ahlfors $Q$-regularity, the $p$-Poincar\'e inequality, and  H\"older's inequality to obtain
\begin{align*}
 |u&_{B(a_{k+2},2^{k+2}|x_0|/(\sigma c_1))}-u_{B(x_1,1)}|\\
 &\le|u_{B(a_{k+2},2^{k+2}|x_0|/(\sigma c_1))}-u_{B(a_{k+2},\frac{|x_1|}{|x_0|})}|+\sum_{i=1}^{\overline{M}-1}|u_{B_i}-u_{B_{i+1}}|+|u_{B(x_1,\frac{|x_1|}{|x_0|})}-u_{B(x_1,1)}|\\
&\lesssim\left(\frac{|x_1|}{|x_0|}\right)^{1-Q/p}\left(\int_{B(a_{k+2},\frac{|x_1|}{|x_0|})}\rho_u^pd\mu\right)^{1/p}+\left(\frac{|x_1|}{|x_0|}\right)^{1-Q/p}\sum_{i=1}^{\overline M}\left(\int_{B_i}\rho_u^pd\mu\right)^{1/p}+|u_{B(x_1,\frac{|x_1|}{|x_0|})}-u_{B(x_1,1)}|\\
&\lesssim \overline M\left(\frac{|x_1|}{|x_0|}\right)^{1-Q/p}\|\rho_u\|_{L^p_\mu(X\setminus B(O,|x_0|/\overline{c_2}))}+|u_{B(x_1,\frac{|x_1|}{|x_0|})}-u_{B(x_1,1)}|.   
\end{align*}
By Theorem~\ref{thm:Morrey}, and the fact that $B(x_1,4|x_1|/|x_0|)\subset X\setminus B(O,|x_0|/(2c_2))$, it follows that 
\begin{align*}
    |u_{B(x_1,\frac{|x_1|}{|x_0|})}-u_{B(x_1,1)}|\lesssim\left(\frac{|x_1|}{|x_0|}\right)^{1-Q/p}\left(\int_{B(x_1,\frac{4|x_1|}{|x_0|})}\rho_u^pd\mu\right)^{1/p}\le\left(\frac{|x_1|}{|x_0|}\right)^{1-Q/p}\|\rho_u\|_{L^p_\mu\left(X\setminus B\left(O,\frac{|x_0|}{2c_2}\right)\right)}.
\end{align*}
Combining this with the previous estimate, we have that 
\[
|u_{B(a_{k+2},2^{k+2}|x_0|/(\sigma c_1))}-u_{B(x_1,1)}|\lesssim\overline M\left(\frac{|x_1|}{|x_0|}\right)^{1-Q/p}\|\rho_u\|_{L^p_\mu\left(X\setminus B\left(O,\frac{|x_0|}{2\overline c_2c_2}\right)\right)}.
\]
Therefore, combining this estimate with \eqref{eq:3008-1}, \eqref{eq:Morrey piece 1}, and \eqref{eq:3008-2} completes the proof.    
\end{proof}

\begin{proof}[Proof of Theorem~\ref{thm:p>Q Growth}]

Let $u\in\dot N^{1,p}(X)$.  By Theorem~\ref{thm:Morrey}, it follows that 
\[
|u(x)-u_{B(x,1)}|\lesssim\left(\int_{B(x,4\lambda)}\rho_u^pd\mu\right)^{1/p}\to 0
\]
as $x\to+\infty$, since $\rho_u$ is $p$-integrable.  Therefore, to prove \eqref{eq:p>Q Growth 1}, it suffices to prove \eqref{eq:p>Q Growth 2} since 
\[
\frac{|u(x)|}{|x|^{1-Q/p}}\le\frac{|u(x)-u_{B(x,1)}|}{|x|^{1-Q/p}}+\frac{|u_{B(x,1)}|}{|x|^{1-Q/p}}.
\]
The proof of \eqref{eq:p>Q Growth 2} follows by the same argument establishing \eqref{equ1.2-2805} in the proof of Theorem~\ref{thm1-2705}, except in this case we use Lemma~\ref{lem:p>Q LogLemma} instead of Lemma~\ref{lem3.4-2805}.
\end{proof}

Using Theorem~\ref{thm2.1-0506}, we obtain modifications of \eqref{eq1.1-2705}, \eqref{equ1.2-2805}, \eqref{eq:p>Q Growth 1}, and \eqref{eq:p>Q Growth 2}, by replacing $\log|x|$ and $|x|$ with quantities of modulus and capacities:

\begin{corollary}Let $1<Q\le p<+\infty$.
Under the assumptions of Theorem~\ref{thm1-2705} when $p=Q$, and under the assumptions of Theorem~\ref{thm:p>Q Growth} when $p>Q$, we have that for every $u\in\dot N^{1,p}(X)$, 
\[
\lim_{x\to+\infty} |u_{B(x,1)}|^p \text{\rm Mod}_p(B(O,1),X\setminus B(O,|x|)) =\lim_{x\to+\infty} |u_{B(x,1)}|^p \text{\rm Cap}_p(B(O,1),X\setminus B(O,|x|)) =0, 	 
 \]
 	and there exists a $Q$-thin set $E$ such that for $p=Q$,
\[
\lim_{E\not\ni x\to+\infty} |u(x)|^Q \text{\rm Mod}_Q(B(O,1),X\setminus B(O,|x|)) =\lim_{E\not\ni x\to+\infty} |u(x)|^Q \text{\rm Cap}_Q(B(O,1),X\setminus B(O,|x|)) =0,
\]
and  that for $p>Q$,
\[
\lim_{ x\to+\infty} |u(x)|^p \text{\rm Mod}_p(B(O,1),X\setminus B(O,|x|)) =\lim_{x\to+\infty} |u(x)|^p \text{\rm Cap}_p(B(O,1),X\setminus B(O,|x|)) =0.
\]
\end{corollary}

\section{Proof of Theorem \ref{thm1.1-3105}}
We now turn our attention to the behaviour at infinity of inhomogeneous Sobolev functions.  We first consider the case when $p>Q$.

\begin{proposition}\label{thm2.2-07-03}
Let $1<Q<+\infty$ and $1<p<+\infty$ so that $Q<p$. Suppose that $(X,d,\mu)$ is a complete, unbounded metric measure space with Ahlfors $Q$-regular measure $\mu$ supporting a $p$-Poincar\'e inequality. Then for every function $u\in  N^{1,p}(X)$, 
\[
	\lim_{x\to +\infty}u(x)=0.
\]
\end{proposition}
\begin{proof}
Let $u\in N^{1,p}(X)$ and let $\rho_u$ be a $p$-integrable upper gradient of $u$.  Since $p>Q$, we have from Theorem~\ref{thm:Morrey} and Ahlfors $Q$-regularity that 
\[
|u(x)-u_{B(x,1)}|\lesssim\left(\int_{B(x,4\lambda)}\rho_u^pd\mu\right)^{1/p}\to 0
\]
as $x\to+\infty$, since $\rho_u$ is $p$-integrable and $X$ is unbounded.  Since $u\in L^p_\mu(X)$, we also have that 
\[
\lim_{x\to+\infty}u_{B(x,1)}=0.
\]
By these two estimates and the triangle inequality, the claim follows.\qedhere
\end{proof}

We now complete the proof of Theorem~\ref{thm1.1-3105}.

\begin{proof}[Proof of Theorem \ref{thm1.1-3105}] When $p>Q$, \eqref{eq:p>Q Inhom} follows from Proposition~\ref{thm2.2-07-03} above. When $p=Q$, by Lemma~\ref{thm1-0705}, there exists a $Q$-thin set $E$ such that
\[
\lim_{E\not\ni x\to+\infty}u(x)= \lim_{E\not\ni x\to+\infty}\dashint_{B(x,|x|/2)}u\,d\mu.
\]
Here the existence of limits follow since $\lim_{x\to+\infty}u_{B(x,|x|/2)}$ exists.
By Ahlfors $Q$-regularity, $\mu(B(x,|x|/2))\to +\infty$ as $x\to+\infty$, and so the right hand side equals $0$ since $u$ is $Q$-integrable.

When $p<Q$, it follows from \cite{KN22} that there exist $c\in\R$ and a family $\Gamma$ of infinite curves with positive $p$-modulus such that $\lim_{t\to+\infty}u(\gamma(t))=c$ for every $\gamma\in\Gamma$.  Suppose that $c\ne 0$.  For each $\gamma\in\Gamma$, there exists an infinite subcurve $\gamma'\subset\gamma$ such that $|u(x)|\ge |c|/2$ for all $x\in \gamma'$.  Let $\Gamma':=\{\gamma':\gamma\in\Gamma\}$.  We then have that 
\[
0<M:=\Mod_p(\Gamma)\le\Mod_p(\Gamma').
\]
For $k\ge 2$, let $A_k:=B(O,2^{k+1})\setminus B(O,2^k)$, and let $E_k:=\bigcup_{\gamma\in\Gamma'}(\gamma\cap A_k)$.  Let $\Gamma_k':=\{\gamma\in\Gamma':\gamma\cap A_{k-1}\ne\varnothing\}$ for $k\ge 1$.  Then, for every $n\ge 2$ we have that 
\[
\Gamma'\subset\bigcup_{k\ge n}\Gamma'_k.
\]
For each $k\ge n$, we have that  for all measurable sets $A\supset E_k$, 
$2^{-k}\chi_{A}$ is admissible for computing the $p$-modulus of $\Gamma'_k$, and so it follows that 
\[
2^{kp}\Mod_{p}(\Gamma'_k)\le\mu(E_k).
\]
Hence, we have that 
\[
\mu(|\Gamma'|)\ge\sum_{k\ge n}\mu(E_k)\ge\sum_{k\ge n}2^{kp}\Mod_p(\Gamma'_k)\ge 2^{np}\Mod_p(\Gamma')\ge2^{np}M,
\]
where $|\Gamma'|$ denotes of the union of the tracjectories of the curves in $\Gamma'$.  Since $n\ge 2$ is arbitrary, we see that $\mu(|\Gamma'|)=+\infty$, and since $|u(x)|\ge |c|/2$ for all $x\in \gamma'$ and for all $\gamma'\in\Gamma'$, we have that $u\not\in L^p_\mu(X)$, a contradiction.  Thus, we have that $c=0$. 
It then follows from the proof of Theorem~\ref{thm:p<Q Thin} that there exists a $p$-thin set $E$ such that $\lim_{E\not\ni x\to+\infty}u(x)=0$.  This gives us \eqref{eq1.4-0506}.

The uniqueness follows from the same argument as in the proof of Theorem~\ref{thm:p<Q Thin} and Theorem~\ref{thm1.3-2208}: if there exists $c'\in\R$ and a $p$-thin set $E'$ such that $\lim_{E'\not\ni x\to+\infty}u(x)=c'$, then by Lemma~\ref{prop2.5-0605} (4.) and (1.), it follows that 
$E\cup E'$ is also a $p$-thin set and $A_{\lambda 2^{j+1}}\setminus (E\cup E')$ is nonempty for sufficiently large $j$.  Thus $c'=0$, completing the proof.
\end{proof}

\section{Examples}
We give the following example by relying  on arguments in \cite{KN22,EKN22,KNW22,K22}.

\begin{example}
\label{example-homogeneous} Let $1<p<+\infty$ and $1<Q<+\infty$ be such that $Q\leq p$. Suppose that $(X,d,\mu)$ is an unbounded metric measure space with metric $d$ and Ahlfors $Q$-regular measure $\mu$. Then there is a function $u\in \dot N^{1,p}(X)$ such that 
	\begin{equation}\label{eq6.1-0306}
	\lim_{t\to+\infty}u(\gamma(t)) \text{\rm \ \ does not exist for any infinite curve $\gamma\in\Gamma^{+\infty}$}
	\end{equation}
and
	\begin{equation}\label{eq6.2-0306}
	\lim_{E\not \ni x\to+\infty}u(x) \text{\rm\ \ does not exist for any $Q$-thin set $E$}.
	\end{equation}
	Here the function $u$ can be chosen bounded or unbounded.
\end{example}
\begin{proof}
Let $O\in X$ be a fixed point.
We will prove this example for the constant $C$, appearing \eqref{eq2.4-2805} in the definition of thin sets, is $1$. The general case follows similarly. 
	Let 
	\[
	 g(x):=\sum_{k=1}^{+\infty} \left(\sum_{j=2^{k}}^{2^{k+1}} \frac{2^{-j}}{2^k}\chi_{A_{2^{j}}}(x) \right)
	\]
	where $A_{ 2^{j}}:=B(O, 2^{j+1})\setminus B(O, 2^{j})$ for $j\in \mathbb N$. Then by Ahlfors $Q$-regularity, we have that for $p\geq Q,$
	\begin{align*}
	\int_{X}g^pd\mu=\sum_{k=1}^{+\infty}\sum_{j=2^k}^{2^{k+1}} \int_{A_{ 2^{j}}} \left(\frac{2^{-j}}{2^k}\right)^pd\mu=\sum_{k=1}^{+\infty}\left(\frac{1}{2^k}\right)^p\sum_{j=2^k}^{2^{k+1}}\frac{\mu(A_{2^j})}{(2^j)^p} &\lesssim \sum_{k=1}^{+\infty} \left(\frac{1}{2^k}\right)^{p}\sum_{j=2^k}^{2^{k+1}}\left(\frac{1}{2^j}\right)^{p-Q}\\
    &\le\sum_{k=1}^{+\infty}\left(\frac{1}{2^k}\right)^{p-1}<\infty,
	\end{align*}
	and so $g$ is $p$-integrable on $X$. We now construct a function with upper gradient $g$ satisfying \eqref{eq6.1-0306} and \eqref{eq6.2-0306}. 
 
Let $\Gamma_{O,x}$ be the collection of all rectifiable curves $\gamma_{O,x}$ connecting $O$ and $x\in X$.  Given $k_1,k_2\in\N$, $k_2>k_1$, let $A_{k_1,k_2}:=B(O,2^{2^{k_2}})\setminus B(O,2^{2^{k_1}})$. If $k_2-k_1\ge 3$ and $k_1\ge 2$, then $A_{k_1,k_2}\supset\bigcup_{k=k_1}^{k_2-2}\bigcup_{j=2^k}^{2^{k+1}}A_{2^j}$, and so for $\gamma\in\Gamma_{O,x}$ with $|x|\ge 2^{2^{k_2}}$, we then have that   
\begin{equation}\label{eq:k_1k_2}
		\int_{\gamma \bigcap A_{k_1,k_2}} gds \ge\sum_{k=k_1}^{k_2-2} \sum_{j=2^k}^{2^{k+1}}\int_{\gamma\bigcap A_{2^j}}\frac{2^{-j}}{2^k}ds \geq k_2-k_1-2.
\end{equation}
 Here the last inequality is  obtained because $\int_{\gamma\bigcap A_{2^j}}ds \geq  \text{\rm diam}(\gamma\bigcap A_{2^j})\geq 2^j$, see for instance \cite[Proposition 5.1.11]{pekka}.
	We first construct a bounded function $u\in \dot N^{1,p}(X)$ satisfying \eqref{eq6.1-0306} and \eqref{eq6.2-0306}.
	Let $\{k_n\}_{n=2}^\infty$ be a sequence such that $k_{n+1}-k_{n}=10$.
	 We define
 \[ u(x)= \begin{cases}
        0 & \text{\rm \ \ if $x\in B(O,2^{2^{k_2}})$}\\
		 \min\left\{ \inf_{\gamma\in \Gamma_{O,x}}\int_{\gamma \bigcap A_{k_{n},k_{n+1}}}gds, 8\right\} & \text{\rm \ \ if $x\in A_{k_n,k_{n+1}}$ and $n\in\mathbb N$ even, }\\
		8-  \min\left\{ \inf_{\gamma\in \Gamma_{O,x}}\int_{\gamma \bigcap A_{k_{n},k_{n+1}}}gds, 8\right\} & \text{\rm \ \ if $x\in A_{k_n,k_{n+1}}$ and $n\in\mathbb N$ odd}.
		\end{cases}
  \]
 Then $g\in L^p_\mu(X)$ is an upper gradient of $u$ by \cite[Page 188-189]{pekka}, and so $u\in \dot N^{1,p}(X)$. If $n$ is even, then by \eqref{eq:k_1k_2}, we have that $5\le u\le 8$ in $A_{k_{n+1}-3,k_{n+1}}$, and similarly $0\le u\le 3$ in $A_{k_{n+1}-3,k_{n+1}}$ if $n$ is odd.  Hence we see that \eqref{eq6.1-0306} is satisfied.  Now, let $E$ be a $Q$-thin set. Since $A_{2^j}\setminus E\neq \emptyset$ for all sufficiently large $j$, by the first claim of Lemma~\ref{prop2.5-0605}, it follows that \eqref{eq6.2-0306} is satisfied.
	
	Next,  we can similarly construct an unbounded function $u$ with a $p$-integrable upper gradient $g$ on $(X,d,\mu)$ satisfying \eqref{eq6.1-0306} and \eqref{eq6.2-0306}. Indeed, 
	let $\{k_n\}_{n=10}^\infty$ be a sequence such that $k_{n+1}-k_{n}=n$, and let $u$ be defined by
	\[ u(x)= \begin{cases}
    0&\text{\rm \ \ if $x\in B(O,2^{2^{k_{10}}})$}\\
 \min\left\{ \inf_{\gamma\in \Gamma_{O,x}}\int_{\gamma \bigcap A_{k_{n},k_{n+1}}}gds, n-5\right\}& \text{\rm \ \ if $x\in A_{k_n,k_{n+1}}$ and $n\in\mathbb N$ even, }\\
		n-6-  \min\left\{ \inf_{\gamma\in \Gamma_{O,x}}\int_{\gamma \bigcap A_{k_{n},k_{n+1}}}gds, n-6\right\} & \text{\rm \ \ if $x\in A_{k_n,k_{n+1}}$ and $n\in\mathbb N$ odd}.
		\end{cases}
	\]
	By a similar argument as in the bounded case, we have that $u\in \dot N^{1,p}(X)$, and $u$ satisfies \eqref{eq6.1-0306} and \eqref{eq6.2-0306}. This completes the proof.\qedhere
	
\end{proof}

In the following examples, we show that we cannot replace the thick set $F$ with an almost thick set in Theorem~\ref{thm1.3-2208}.  We first give an easy example where we fail to have uniqueness of limits along almost thick sets.

\begin{example}\label{ex:EasyAlmostThick}
Let $1<Q<+\infty$, and suppose that $(X,d,\mu)$ is a complete, unbounded metric space with Ahlfors $Q$-regular measure $\mu$ supporting a $Q$-Poincar\'e inequality.  Then there exist almost thick sets $F_1,F_2\subset X$ and a function $u\in\dot N^{1,Q}(X)$ such that $\lim_{x\to+\infty,x\in F_1}u(x)=1$ and $\lim_{x\to+\infty,x\in F_2}u(x)=0$.
\end{example}

\begin{proof}
For each $j\in\N$ let $A_j:=B(O,2^{j+1})\setminus\overline B(O,2^j)$, and choose $x_j, x_j'\in A_j$ such that $B(x_j,2^{j-10})\Subset A_j$, $B(x_j',2^{j-10})\Subset A_j$ and $B(x_j, 2^{j-10})\cap B(x_j',2^{j-10})=\emptyset$.  We can find such  $x_j, x_{j}'$ because $X$ supports a Poincar\'e inequality, and is thus connected.  Let $0<r_j<2^{j-20}$ and let $B_j:=B(x_j,r_j), B_j':=B(x_j',r_j)$.  By Theorem~\ref{thm2.1-0506}, we have that 
\[
\text{\rm Cap}_Q(B_j,X\setminus B(x_j,2^{j-10}))\lesssim\log(2^{j-10}/r_j)^{1-Q},
\]
and so there exists $u_j\in \dot N^{1,Q}(X)$ with $u_j\equiv 1$ on $B_j$ and $u_j\equiv 0$ on $X\setminus B(x_j,2^{j-10})$ such that 
\[
\int_{X}g_{u_j}^Qd\mu\lesssim\log(2^{j-10}/r_j)^{1-Q}.
\]
Letting $u=\sum_{j\in\mathbb N} u_j$, which has $g=\sum_{j\in\mathbb N} g_{u_j}$ as an upper gradient, we have that
\[
\int_X g^Qd\mu\lesssim\sum_{j\in\mathbb N}\log(2^{j-2}/r_j)^{1-Q}<\infty
\]
provided $r_j$ is small enough.  Choosing $r_j=2^{j-20}e^{-2^j}$ achieves this, and so $u\in\dot N^{1,Q}(X)$ with this choice.  Letting $F_1=\bigcup_j B_j$, we see that $u\equiv 1$ on $F_1$.  Furthermore, using Lemma~\ref{lem:Measure-Content Bound} we have that 
\begin{align*}
   \Ha^\alpha_{+\infty}(F_1\cap A_j)=\Ha^\alpha_{+\infty}(B_j)\gtrsim\frac{\mu(B_j)}{(2^j)^{Q-\alpha}}\simeq \frac{r_{j}^Q 2^{j\alpha}}{2^{jQ}}\simeq (e^{-2^j})^Q2^{j\alpha},   
\end{align*}
for all $0<\alpha<Q$, and so $F_1$ is almost thick, with $\delta_j=(e^{-2^j})^Q$.  
Repeating the above argument, 
we have that $F_2:=\cup_j B_j'$ is almost thick with $\delta_j$ depending on the radii of the $B_j'$, and we see that $u\equiv 0$ on $F_2$.   
\end{proof}

 \begin{example}
 \label{example5.4-0709} Let $1<Q<+\infty$. Suppose that $(X,d,\mu)$ is a complete, unbounded metric measure space with metric $d$ and Ahlfors $Q$-regular measure $\mu$ supporting a $Q$-Poincar\'e inequality.  Then  there is an almost thick set $F$ which is not thick and a function $f\in \dot N^{1,Q}(X)$ such that
 \begin{equation}\label{eq6.3-0809}
 	\text{\rm  $\lim_{F\ni x\to+\infty}f(x)$ exists but\ \ }\lim_{t\to+\infty} f(\gamma(t)) \text{\rm \ \ does not exist for any $\gamma\in\Gamma^{+\infty}$.}
 \end{equation}
 \end{example}
 
 \begin{proof}Let $O\in X$ be a fixed point, and let 
%
%
$\{r_j\}_{j\in\mathbb N}, \{s_j\}_{j\in\mathbb N}$ be two sequences defined by 
 \[
 	r_1=2, s_1=2^{2}, s_{j}=2^{r_{j}}, r_{j+1}=2^{s_j} \ \ \text{\rm for  $j\in\mathbb N$}.
 \]
 Then 
 \begin{equation}\label{eq6.4-0809}
 	100 r_j<s_{j}<r_{j+1}/100 \ \ \text{\rm and \ \ } \lim_{j\to+\infty} \frac{s_j^\alpha}{r_{j+1}^\alpha}=0 \text{\rm \ \ for all $\alpha\in(0,Q)$}.
 \end{equation}
 We let 
\[\text{
$L_j:=\#\{i\in\mathbb N: r_j\leq 2^i\leq s_j/2\}$ \ and \  $N_j:=\#\{i\in\mathbb N: s_j\leq 2^i\leq r_{j+1}\}$
}
\]
for $j\in\mathbb N$.
 Then $L_j\geq 2^j$ and $N_j\geq 2^j$. We set 
 \[
 g_{s_{j}, r_j}:=  \sum_{i\in\mathbb N, r_j\leq 2^i\leq s_j/2} \frac{2^{-i}}{L_j} \chi_{B(O,2^{i+1})\setminus B(O,2^{i})} \text{\rm \ \ and \ \ } g_{r_{j+1}, s_j}:= \sum_{i\in\mathbb N, s_j\leq 2^i\leq r_{j+1}} \frac{2^{-i}}{N_j}  \chi_{B(O,2^{i+1})\setminus B(O,2^{i})}
 \]
for $j\in\mathbb N$.
 We have that $g:=\sum_{j\in\mathbb N} (g_{s_{j}, r_j}+  g_{r_{j+1}, s_j})$ belongs to $L_\mu^Q(X)$ because 
 \begin{align*}
 \int_Xg^Qd\mu=& \sum_{j\in\mathbb N}  \sum_{i\in\mathbb N, r_j\leq 2^i\leq s_j/2} \left(\frac{2^{-i}}{L_j}\right)^Q \int_{ B(O,2^{i+1})\setminus B(O,2^{i})}d\mu +  \sum_{j\in\mathbb N}  \sum_{i\in\mathbb N, s_j\leq 2^i\leq r_{j+1}} \left(\frac{2^{-i}}{N_j}\right)^Q \int_{ B(O,2^{i+1})\setminus B(O,2^{i})}d\mu \\
 \approx & \sum_{j\in\mathbb N} \frac{1}{L_j^Q}  \sum_{i\in\mathbb N, r_j\leq 2^i\leq s_j/2}  1 + \sum_{j\in\mathbb N}\frac{1}{N_j^Q}  \sum_{i\in\mathbb N, s_j\leq 2^i\leq r_{j+1}}  1\\
 =& \sum_{j\in\mathbb N} \left(\frac{1}{L_j^{Q-1}}+\frac{1}{N_j^{Q-1}}\right)\leq \sum_{j\in\mathbb N} \left(\frac{1}{(2^{j})^{Q-1}} +\frac{1}{(2^j)^{Q-1}}\right)<+\infty. 
 \end{align*}
 Let $\Gamma_{O,x}$ be the collection of all rectifiable curves $\gamma_{O,x}$ connecting $O$ and $x$ where $x\in X$. We define a function $u$ by setting
 \[	
 u(x)=
 \begin{cases}
 	\min\left\{1, \inf_{\gamma\in\Gamma_{O,x}} \int_{\gamma} g_{s_j,r_j}ds \right\}	 & \text{\rm \ \ if \ \ $x\in B(O,s_j/2)\setminus B(O,r_j)$},\\ 
	1& \text{\rm \ \ if \ \ $x\in B(O,s_j)\setminus B(O,s_j/2)$},\\
	1- \min\left\{1, \inf_{\gamma\in\Gamma_{O,x}} \int_{\gamma} g_{r_{j+1},s_j}ds \right\}	& \text{\rm \ \ if \ \ $x\in B(O,r_{j+1})\setminus B(O,s_j)$},
\end{cases}
 \]
for each $j\in\mathbb N$.
 Then $g$ is an upper gradient of $u$ and hence $u\in\dot N^{1,Q}(X)$.  Moreover, $u(x)=0$ for all $x$ such that $d(O,x)=r_j$.  

 For each $10\leq j\in\mathbb N$, we pick $x_j\in X$ so that $B(x_j,2^{-j})\subset B(O,2^{j+1})\setminus B(O,2^j))$, and so that 
 \begin{equation}\label{eq:r_i}B(x_j, 2^{-j})\cap\left(\bigcup_i\{x:|x|=r_i\}\right)=\varnothing.
 \end{equation}
  Since our space is connected by the Poincar\'e inequality,  then we  can find such an $x_j$   by the definition of the $r_i$.  For each $j\ge 10$, define 
 $$B_j:=B(x_j,2^{-j^{\frac{2}{Q-1}}}\times 2^{-j}),$$
 and let $F:=\bigcup_{j=10}^{+\infty}B_j$. By Lemma~\ref{lem:Measure-Content Bound} and Ahlfors $Q$-regularity, we have that  
\begin{align*}
\frac{\mathcal H^{\alpha}_{+\infty}(F\cap B(O,2^{j+1})\setminus B(O,2^j))}{(2^j)^\alpha}=\frac{\Ha^{\alpha}_{+\infty}(B_j)}{(2^j)^\alpha}\gtrsim\frac{\mu(B_j)}{(2^j)^Q}\simeq\frac{(2^{-j^{\frac{2}{Q-1}}}\times2^{-j})^Q}{(2^j)^Q}=:\delta_j\to 0
\end{align*}
as $j\to+\infty$ for all $0<\alpha<Q$.  Thus, $F$ is almost thick as in Definition~\ref{def:almost-thick}. However, $F$ is not thick since 
\[
\frac{\mathcal H^{\alpha}_{+\infty}(F\cap B(O,2^{j+1})\setminus B(O,2^j))}{(2^j)^\alpha}=\frac{\Ha^{\alpha}_{+\infty}(B_j)}{(2^j)^\alpha}\le \frac{(2^{-j^{\frac{2}{Q-1}}}\times2^{-j})^\alpha}{(2^j)^\alpha}\to 0
\]
as $j\to+\infty$ for all $0<\alpha<Q$.

For each $j\geq 10$, we use Theorem~\ref{thm2.1-0506} to choose a function $0\le\psi_j\le 1$ which is admissible for computing $\text{\rm Cap}_Q(B_j, X\setminus B(x_j,2^{-j}))$ such that $$\int_Xg_{\psi_j}^Qd\mu\lesssim \left( \log \left(\frac{1}{2^{-j^{\frac{2}{Q-1}}}}\right)\right)^{1-Q}\approx j^{-2},$$ where $g_{\psi_j}$ is the minimal $Q$-weak upper gradient of $\psi_j$.  By this estimate and the fact that each $\psi_j$ is supported in $B(x_j,2^{-j})$, it follows that $\sum_{j=10}^{+\infty}\psi_j\in  N^{1,Q}(X)$. 

Set $f:=\min\{1, u+\sum_{j=10}^{+\infty}\psi_j\}$.  Since $\psi_j\equiv 1$ on $B_j$, it follows that $f|_F\equiv 1$, and by \eqref{eq:r_i}, we have that $f|_{\{x:|x|=r_i\}}\equiv 0$ for all $i$.  Therefore \eqref{eq6.3-0809} follows, completing the proof. 
 \end{proof}

\begin{example}
	\label{example-p>n}Let $n\in\mathbb N$ be so that $n\geq 2$. We denote $\dot W^{1,p}(\mathbb R^n\setminus B(O,1))$  the classical homogeneous $p$-Sobolev space where $n<p<+\infty$ and $B(O,1)$  the ball with radius $1$ and center at the origin $O$. Then
 there exists a continuous function $u\in \dot W^{1,p}(\mathbb R^n\setminus B(O,1))$ such that its limit at infinity  along all radial curves exists but is not unique.
\end{example}
\begin{proof}
	Let $x=(x_1,x_2,\ldots,x_n)\in\mathbb R^n\setminus B(O,1)$. We define 
		$u(x)=\frac{x_n}{|x|} \text{\rm \ \ for  $x\in\mathbb R^n\setminus B(O,1)$}
		$
	 where $|x|^2=\sum_{i=1}^nx^2_i$.
 Since $||x|^2-x_n^2|\leq |x|^2$ and $|x_nx_i|\lesssim |x|^2$ for $1\le i\le n-1$, we obtain that for $p>n$,
		\begin{align*}
		\int_{\mathbb R^n\setminus B(O,1)}|\nabla u|^pdx
		= \int_{\mathbb R^n\setminus B(O,1)} \left ({\left | \frac{|x|^2-x_n^2}{|x|^{{3}}}\right |^2+\sum_{i=1}^{n-1}\left |\frac{x_nx_i}{|x|^{{3}}}\right |^2}\right )^{\frac{p}{2}}dx
		\lesssim  \int_{\mathbb R^n\setminus B(O,1)}\frac{1}{|x|^p}dx<{+\infty}.
		\end{align*}
	Let $x=(r,\xi_1,\xi_2,\ldots,\xi_{n-1})$ in the spherical coordinate.  That is, 
	\[\begin{cases}
		x_1=r\cos(\xi_1),\\
		 x_{i}=r\sin(\xi_1)\sin(\xi_2)\ldots\sin(\xi_{i-2})\cos(\xi_{i-1}), \rm where\ 2\leq i\leq n-1,\\
		  x_{n}=r\sin(\xi_1)\sin(\xi_2)\ldots\sin(\xi_{n-2})\sin(\xi_{n-1}),
	\end{cases}
	\]
	where $\xi_{n-1}\in[0,2\pi)$, $\xi_{i}\in[0,\pi]$ for  $1\le i\le n-2$.
	Hence for $\xi:=(\xi_1,\xi_2,\ldots,\ldots,\xi_{n-1})$ in the unit sphere,
	$\lim_{r\to{+\infty}}u(r\xi)=\sin(\xi_1)\sin(\xi_2)\ldots\sin(\xi_{n-2})\sin(\xi_{n-1}).
	$
	
\end{proof}

\section*{Acknowledgements}
J.K. has been supported by NSF Grant\#DMS-2054960 and the University of Cincinnati's University Research Center summer grant.  P.K. has been supported by the Academy of Finland Grant number 364210.
K.N. has been supported by the National Natural Science Foundation of China (No. 12288201).  Part of this research was conducted while J.K. was visiting the University of Jyv\"askyl\"a during the Spring of 2023, and he would like to thank this institution for its generous support and hospitality.
{
We thank the anonymous referees for their careful reading, thoughtful comments, for pointing out an unnoticed error, and
particularly for emphasizing the notion of polar coordinates, which  improved the presentation of the paper.
}

\vskip.2cm

\noindent{\bf Data Availability:} No datasets were generated or analyzed during this study.

\end{document}